\documentclass{amsart}

\usepackage{amssymb}
\usepackage[utf8]{inputenc}
\usepackage{xcolor}
\usepackage{mathtools}
\usepackage{hyperref}

\newcommand{\Aff}{\operatorname{Aff}}
\newcommand{\Z}{\mathbb{Z}}
\newcommand{\Aut }{\mathrm{Aut}}

\newcommand{\cB }{\mathcal{B}}
\newcommand{\cF }{\mathcal{F}}

\newcommand{\cX }{\mathcal{X}}

\newcommand{\End }{\mathrm{End}}
\newcommand{\fie }{\Bbbk }

\newcommand{\fm }{\mathfrak{m}}

\newcommand{\gr }{\mathrm{gr}}
\newcommand{\id }{\mathrm{id}}

\newcommand{\NA}{\mathcal{B}}
\newcommand{\cD}{\mathcal{D}}
\newcommand{\cA}{\mathcal{A}}

\newcommand{\ndN}{\mathbb{N}}

\newcommand{\ndZ}{\mathbb{Z}}
\newcommand{\ot }{\otimes }

\newcommand{\yd}[1]{\prescript{#1}{#1}{\mathcal{YD}}}

\newcommand{\ad}{\operatorname{ad}}
\newcommand{\ch}{\mathrm{char}}
\newcommand{\supp}{\operatorname{supp}}

\newcommand{\Dchaintwo}[3]{%
        \setlength{\unitlength}{.08cm}
        \rule[-3\unitlength]{0pt}{8\unitlength}
        \begin{picture}(14,5)(0,3)
                \put(1,2){\circle{2}}
                \put(2,2){\line(1,0){10}}
                \put(13,2){\circle{2}}
                \put(1,5){\makebox[0pt]{\scriptsize #1}}
                \put(7,4){\makebox[0pt]{\scriptsize #2}}
                \put(13,5){\makebox[0pt]{\scriptsize #3}}
        \end{picture}
}

\numberwithin{equation}{section}

\newtheorem{thm}{Theorem}[section]
\newtheorem{lem}[thm]{Lemma}
\newtheorem{pro}[thm]{Proposition}
\newtheorem{cor}[thm]{Corollary}
\newtheorem{rem}[thm]{Remark}
\newtheorem{question}[thm]{Question}

\theoremstyle{definition}
\newtheorem{defi}[thm]{Definition}
\newtheorem{exa}[thm]{Example}

\title[Nichols algebras of simple Yetter--Drinfeld modules]{Finite-dimensional Nichols algebras of simple
Yetter--Drinfeld modules (over groups) of prime dimension}

\author{I. Heckenberger}
\author{E. Meir}
\author{L. Vendramin}

\address[I. Heckenberger]{Philipps-Universit\"at Marburg,
        FB Mathematik und Informatik,
        Hans-Meer\-wein-Stra\ss e,
35032 Marburg, Germany.}
\email{heckenberger@mathematik.uni-marburg.de}

\address[E. Meir]{Institute of Mathematics, University of Aberdeen, Fraser Noble Building, Aberdeen AB24 3UE, UK}
\email{meirehud@gmail.com}

\address[L. Vendramin]{Department of Mathematics and Data
Science, Vrije Universiteit Brussel, Pleinlaan 2, 1050 Brussel}
\email{Leandro.Vendramin@vub.be}


\begin{document}

\begin{abstract}
    Over fields of characteristic zero, we determine all absolutely irreducible 
    Yetter--Drinfeld modules over groups that have prime dimension and yield a finite-dimensional Nichols algebra. To achieve our goal, 
    we introduce orders of braided vector 
    spaces and study their degenerations and specializations. 
\end{abstract}

\maketitle

\section{Introduction}
\label{ss:introduction}

Pointed Hopf algebras form a central class of objects in the 
theory of Hopf algebras since the beginning \cite{MR0252485}. 
Besides the coradical of a pointed Hopf algebra, which is the group ring of a group, the skew-primitive elements and the infinitesimal braiding belong to the most important invariants \cite{MR1659895}. The infinitesimal braiding $V$ is a
braided vector space that yields another invariant: a connected strictly graded braided Hopf algebra known as the 
\emph{Nichols algebra} of~$V$. Very prominent examples are the positive part of quantum groups, but the structure theory is also deeply understood for abelian groups and corresponding braided vector spaces of diagonal type. For a general account, we refer to 
\cite{MR3728608} and \cite{MR4164719}.

Nichols algebras also appear in the highly influential papers of Nichols \cite{MR506406}, Woronowicz \cite{MR901157,MR994499} and  
Majid \cite{MR1969778}, Schaubenburg \cite{MR1396857}, Rosso \cite{MR1632802}, Kharchenko \cite{MR1763385}, 
and Andruskiewitsch and Schneider \cite{MR1913436}. Recent interest and applications of 
Nichols algebras appear in algebraic geometry in the work of Kapranov and Schechtman \cite{MR4176533},
and in quantum field theory 
in the paper of Lentner \cite{MR4184294}.  In number theory,  
Ellenberg, Tran and Westerland
\cite{ETW} used Nichols algebras to prove an upper bound in the weak Malle
conjecture on the distribution of finite extensions of $\mathbb{F}_q(t)$ 
with specified Galois groups. 

The classification of finite-dimensional Nichols algebras
of Yetter--Drinfeld modules
was achieved in \cite{MR2462836} for abelian groups and fields of characteristic zero, and
in \cite{MR3605018,MR3656477} for non-abelian groups and 
semi-simple non-simple Yetter--Drinfeld modules over arbitrary fields. The key structure in these classifications is the Weyl groupoid \cite{MR2766176,MR2207786}. 
For most of the Nichols algebras in the classification, a construction from Nichols algebras of diagonal type by folding was described by Lentner \cite{MR3253277}.
However, the classification problem is wide open for irreducible Yetter--Drinfeld modules, despite of tremendous efforts taken by different authors using different tools \cite{MR3395052,MR3493214,MR3713037,MR4109132,MR4325965,MR2786171,
MR2745542,MR2209265,MR3981991,
MR3552907,MR4151588,MR2891215,MR3356939,MR4034793}. Over non-abelian groups, not even a satisfactory unified explanation of the Hilbert series of the known finite-dimensional examples is available. At the moment, only a few finite-dimensional
Nichols algebras are known. 

For any group $G$, let $\yd{\fie G}$ denote the category of Yetter--Drinfeld modules
over the group ring $\fie G$ of $G$.

\begin{exa}
\label{exa:dim12}
    Let $G$ be a non-abelian epimorphic image of the group
    \[
    \langle x_1,x_2,x_3:x_1x_2=x_3x_1,\,x_1x_3=x_2x_1,\,x_2x_3=x_1x_2\rangle. 
    \]
    For all $i\in\{1,2,3\}$, let $g_i\in G$ be the image of $x_i$. Then $g_i\ne g_j$ for all $i\ne j$. 
    Let $\langle g_1\rangle $ be the subgroup of $G$ generated by $g_1$,
    and let $U$ be a one-dimensional Yetter--Drinfeld module over $\langle g_1\rangle $
    with $\fie \langle g_1\rangle $-coaction $\delta(u)=g_1\ot u$ for all $u\in U$
    and $\fie \langle g_1\rangle $-action $g_1u=-u$ for all $u\in U$.
    Then $V=\fie G\ot_{\fie \langle g_1\rangle}U$ is a Yetter--Drinfeld module over $\fie G$.
    Let $v_1=1\ot u$ for some $u\in U$ with $u\ne 0$. Then
    \[
        v_1,\,v_2=-g_3v_1,\,v_3=-g_2v_1
    \]
    form a basis of $V$. The $G$-degrees of these 
    vectors are $g_1$, $g_2$ and $g_3$, respectively. The action of $G$ on $V$
    is given by 
    \[
    g_i v_j=-v_{2i-j\bmod 3},\qquad i,j\in\{1,2,3\}.
    \]
    The support of $V$ is the affine rack $\Aff(3,2)$, see
    Section~\ref{ss:preliminaries}.
    Then $\dim\NA(V)=12$. This example appeared first
    in~\cite{MR1800714}; see also \cite{MR1667680}.  
\end{exa}

The following two examples appeared  first in~\cite{MR1994219}.

\begin{exa}
\label{exa:dim1280a}
Let $G$ be a non-abelian epimorphic image of the group
\[
\langle x_1,x_2,x_3,x_4,x_5:x_ix_jx_i^{-1}=x_{-i+2j\bmod 5}\rangle.
\]
For all $i\in\{1,2,3,4,5\}$, let 
$g_i\in G$ be the image of $x_i$.
Then $g_i\ne g_j$ for all $i\ne j$.
Let $\langle g_1\rangle $ be the subgroup of $G$ generated by $g_1$,
and let $U$ be a one-dimensional Yetter--Drinfeld module over $\langle g_1\rangle $
with $\fie \langle g_1\rangle $-coaction $\delta(u)=g_1\ot u$ for all $u\in U$
and $\fie \langle g_1\rangle $-action $g_1u=-u$ for all $u\in U$.
Then $V=\fie G\ot_{\fie \langle g_1\rangle}U$ is a Yetter--Drinfeld module over $\fie G$.
Let $v_1=1\ot u$ for some $u\in U$ with $u\ne 0$. 
Then the vectors 
    \[
        v_1,\,v_2=-g_5v_1,\,v_3=-g_4v_1,\,v_4=-g_3v_1,\,v_5=-g_2v_1
    \]
    form a basis of $V$. 
    For each $i\in\{1,2,3,4,5\}$, $\deg v_i=g_i$. The action of $G$ on $V$
    is given by 
    \[
    g_i v_j=-v_{-i+2j\bmod 5},\qquad i,j\in\{1,2,3,4,5\}.
    \]
    The support of $V$ is the affine rack $\Aff(5,2)$. 
    Then $\dim\NA(V)=1280$. 
\end{exa}

\begin{exa}
\label{exa:dim1280b}
Let $G$ be a non-abelian epimorphic image of the group
\[
\langle x_1,x_2,x_3,x_4,x_5:x_ix_jx_i^{-1}=x_{3(i+j)\bmod 5}\rangle.
\]
For all $i\in\{1,2,3,4,5\}$, let 
$g_i\in G$ be the image of $x_i$.
 Then $g_i\ne g_j$ for all $i\ne j$.
Let $\langle g_1\rangle $ be the subgroup of $G$ generated by $g_1$,
and let $U$ be a one-dimensional Yetter--Drinfeld module over $\langle g_1\rangle $
with $\fie \langle g_1\rangle $-coaction $\delta(u)=g_1\ot u$ for all $u\in U$
and $\fie \langle g_1\rangle $-action $g_1u=-u$ for all $u\in U$.
Then $V=\fie G\ot_{\fie \langle g_1\rangle}U$ is a Yetter--Drinfeld module over $\fie G$.
Let $v_1=1\ot u$ for some $u\in U$ with $u\ne 0$. Then the vectors 
    \[
        v_1,\,v_2=-x_3v_1,\,v_3=-x_5v_1,\,v_4=-x_2v_1,\,v_5=-x_4v_1
    \]
    form a basis of $V$. 
    For each $i\in\{1,2,3,4,5\}$, $\deg v_i=g_i$. The action of $G$ on $V$
    is given by 
    \[
    x_i v_j=-v_{3(i+j)\bmod 5},\qquad i,j\in\{1,2,3,4,5\}.
    \]
    The support of $V$ is the affine rack $\Aff(5,3)$. 
    Then $\dim\NA(V)=1280$. 
\end{exa}

Gra\~na found the following two examples. 

\begin{exa}
\label{exa:dim326592a}
Let $G$ be a non-abelian epimorphic image of the group
\[
\langle x_1,x_2,\dots,x_7:x_ix_jx_i^{-1}=x_{5i+3j\bmod 7}\rangle.
\]
For all $i\in\{1,2,\dots,7\}$, let 
$g_i\in G$ be the image of $x_i$.
 Then $g_i\ne g_j$ for all $i\ne j$.
Let $\langle g_1\rangle $ be the subgroup of $G$ generated by $g_1$,
and let $U$ be a one-dimensional Yetter--Drinfeld module over $\langle g_1\rangle $
with $\fie \langle g_1\rangle $-coaction $\delta(u)=g_1\ot u$ for all $u\in U$
and $\fie \langle g_1\rangle $-action $g_1u=-u$ for all $u\in U$.
Then $V=\fie G\ot_{\fie \langle g_1\rangle}U$ is a Yetter--Drinfeld module over $\fie G$.
Let $v_1=1\ot u$ for some $u\in U$ with $u\ne 0$. 
Then the vectors 
    \[
        v_1,\,v_2=-g_4v_1,\,v_3=-g_7v_1,\,
        v_4=-g_3v_1,\,
        v_5=-g_6v_1,\,
        v_6=-g_2v_1,\,
        v_7=-g_5v_1,
    \]
    form a basis of $V$. 
    For each $i\in\{1,2,\dots,7\}$, $\deg v_i=g_i$.
    The action of $G$ on $V$
    is given by 
    \[
    x_i v_j=-v_{5i+3j\bmod 7},\qquad i,j\in\{1,2,\dots,7\}.
    \]
    The support of $V$ is the affine rack $\Aff(7,3)$. 
    Then $\dim\NA(V)=326592$. 
\end{exa}

\begin{exa}
\label{exa:dim326592b}
Let $G$ be a non-abelian epimorphic image of the group
\[
\langle x_1,x_2,\dots,x_7:x_ix_jx_i^{-1}=x_{3i+5j\bmod 7}\rangle.
\]
For all $i\in\{1,2,\dots,7\}$, let 
$g_i\in G$ be the image of $x_i$.
 Then $g_i\ne g_j$ for all $i\ne j$.
Let $\langle g_1\rangle $ be the subgroup of $G$ generated by $g_1$,
and let $U$ be a one-dimensional Yetter--Drinfeld module over $\langle g_1\rangle $
with $\fie \langle g_1\rangle $-coaction $\delta(u)=g_1\ot u$ for all $u\in U$
and $\fie \langle g_1\rangle $-action $g_1u=-u$ for all $u\in U$.
Then $V=\fie G\ot_{\fie \langle g_1\rangle}U$ is a Yetter--Drinfeld module over $\fie G$.
Let $v_1=1\ot u$ for some $u\in U$ with $u\ne 0$. 
Then the vectors 
    \[
        v_1,\,v_2=-g_6v_1,\,v_3=-g_4v_1,\,
        v_4=-g_2v_1,\,
        v_5=-g_7v_1,\,
        v_6=-g_5v_1,\,
        v_7=-g_3v_1,
    \]
    form a basis of $V$. The degrees of these 
    vectors are $x_1,x_2,\dots,x_7$, respectively. The action of $G$ on $V$
    is given by 
    \[
    x_i v_j=-v_{3i+5j\bmod 7},\qquad i,j\in\{1,2,\dots,7\}.
    \]
    The support of $V$ is the affine rack $\Aff(7,5)$. 
    Then $\dim\NA(V)=326592$. 
\end{exa}

Examples~\ref{exa:dim1280a}--\ref{exa:dim326592b} 
in arbitrary characteristic were discussed in~\cite{MR2803792}.  
We collected some information on these Nichols algebras in~Table \ref{tab:nichols}, where we write  
\[
(n)_t=1+t+\cdots+t^{n-1}\in \ndZ [t] 
\]
for all integers $n\ge 0$.

Be warned, that there exists at least one additional finite-dimensional Nichols algebra
over $\Aff(3,2)$ in characteristic two; see \cite[Appendix A]{MR2891215}. Hence Table~\ref{tab:nichols} should not be regarded as a complete list
of finite-dimensional Nichols algebras over simple racks with prime cardinality. 

\begin{table}[ht]
\caption{Finite-dimensional Nichols algebras over affine racks (in arbitrary characteristic).}
\begin{center}
\label{tab:nichols}
\begin{tabular}{|c|c|c|c|c|c}
\hline
$\dim V$ & $\supp V$ & $\dim\NA(V)$ & Hilbert series & Comments\tabularnewline
\hline
3 & $\Aff(3,2)$ & $12$ & $(2)^2_t (3)_t$ & Example~\ref{exa:dim12}\\
\hline
5 & $\Aff(5,2)$ & $1280$ & $(4)^4_t (5)_t$ & Example~\ref{exa:dim1280a}\\
\hline
5 & $\Aff(5,3)$ & $1280$ & $(4)^4_t (5)_t$ & Example~\ref{exa:dim1280b}\\
\hline
7 & $\Aff(7,3)$ & $326592$ & $(6)^6_{t} (7)_t$ & Example~\ref{exa:dim326592a}\\
\hline
7 & $\Aff(7,5)$ & $326592$ & $(6)^6_{t} (7)_t$ & Example~\ref{exa:dim326592b} \\
\hline
\end{tabular}
\end{center}
\end{table}

With the present paper, we initiate the study of deformations of braided vector spaces to retrieve additional structural information about Nichols algebras of simple Yetter--Drinfeld modules. We introduce and study braided vector space filtrations and orders of braided vector spaces. Note that
filtrations of Yetter--Drinfeld modules are already commonly used, see e.g. \cite{MR4298502}. For simple Yetter--Drinfeld modules however those filtrations are trivial, whereas our definition is much less restrictive. This fact is one of the main reasons that our approach leads to new results.

Deformation techniques are common in several parts of mathematics. Even in the representation theory of generalized quantum groups, (Poisson) orders have been used successfully by Angiono, Andruskiewitsch and Yakimov in \cite{poisson}. In our context, orders of braided vector spaces and special properties in suitable singular points are studied to obtain information on the size of the Nichols algebra in characteristic $0$. This is a new approach, as orders have not yet been applied systematically in the structure theory of Nichols algebras of braided vector spaces.
We demonstrate the method's power by solving the long-standing problem of classifying finite-dimensional Nichols algebras of irreducible Yetter--Drinfeld modules of prime dimension over non-abelian groups. We confirm that there exist no finite-dimensional Nichols algebras in this class except the known examples:
Over fields of characteristic $0$, 
the Nichols algebras of Examples \ref{exa:dim12}--\ref{exa:dim326592b}
are the only examples appearing when the (absolutely irreducible) 
braided vector space has prime dimension. 

\begin{thm}
\label{thm:main}
Assume that $\ch(\fie)=0$. 
Let $V$ be an absolutely irreducible Yetter--Drinfeld module
over a group $G$ such that the support 
$\{x\in G:V_x\ne 0\}$ of $V$ 
generates
$G$. Assume that $\dim V$ is a prime number. Then
$\NA(V)$ is finite-dimensional if and only if
$\dim V\in\{3,5,7\}$ and $V$ is isomorphic to one of the Yetter--Drinfeld modules
of Examples~\ref{exa:dim12}, \ref{exa:dim1280a}, \ref{exa:dim1280b},
~\ref{exa:dim326592a},~\ref{exa:dim326592b}. 
\end{thm}

%
We remark that in Theorem~\ref{thm:main} the assumption on the group being generated by the support of $V$ 
is not too restrictive. Indeed, for the study of $\NA(V)$ one can always  
replace $G$ by its subgroup generated by the support of $V$. 

In \cite{MR1800714}, Milinski and Schneider studied
Nichols algebras of Yetter--Drinfeld modules over 
Coxeter groups. Later, in \cite{MR1714540}, 
Andruskiewitsch and Gra\~na explicitly considered the case of  
Nichols algebras over dihedral groups
of size $2p$, where $p$ is a prime number. 
In the non-abelian case, 
this boils down
to study Nichols algebras
of braided vector spaces 
of dihedral type, that is, finite-dimensional  
vector spaces $V_p$, where $p\geq3$ is a prime number and  
the braiding of $V_p$ is of the form 
\[
c\colon V_p\otimes V_p\to V_p\otimes V_p,\quad 
c(v_i\otimes v_j)=\lambda v_{2i-j\bmod p}\otimes v_i,
\]
for some basis $v_0,v_1,\dots,v_{p-1}$ and 
some non-zero scalar $\lambda$. 

The case where $(p,\lambda)=(3,-1)$ is that of Example~\ref{exa:dim12}. 

The question of determining the finite-dimensional Nichols algebras 
whose support is a simple rack
was raised in \cite[page 228]{MR2786171} and discussed in several papers, including 
\cite{MR2799090,MR2745542,MR3077241} 
and \cite[\S4.4]{MR3751453}.
Theorem~\ref{thm:main} answers this question for simple racks with a prime cardinal. 


The paper is organized as follows. In Section \ref{ss:preliminaries} we discuss
affine racks and absolutely irreducible Yetter--Drinfeld modules of prime dimension. 
Section \ref{ss:orders} concerns orders of braided vector spaces. In Section \ref{ss:positive}
we discuss the passage from characteristic zero to positive characteristic. The proof
of Theorem \ref{thm:main} appears in Section \ref{ss:proof}. 
In an appendix, we discuss the applicability of our methods to similar problems 
like that of Nichols algebras over symmetric groups or classes of non-abelian finite 
groups (see Corollary \ref{cor:simples}).


%

\subsection*{The strategy of the proof}

The proof of Theorem \ref{thm:main} consists of three major steps.
First, we study Nichols algebras of Yetter--Drinfeld orders and, in particular, their
specializations at well-chosen primes.
Second, we relate these specializations to Nichols algebras of diagonal type using filtrations of braided vector spaces. Finally, we use the classification of rank two Nichols algebras of diagonal type in positive characteristic obtained in \cite{MR3313687}. To deal with the cases not treated by this method, we develop a second technique based on Yetter--Drinfeld orders which relies on the classification of finite-dimensional Nichols algebras of semisimple Yetter--Drinfeld modules in positive characteristic in \cite{MR3656477}.


\section{Preliminaries}
\label{ss:preliminaries}

In this section, we review some basic notions about racks and quandles. 
We refer to \cite{MR1994219} for more details. 
A \emph{rack} is a set $X$ together with a binary operation $(x,y)\mapsto x\triangleright y$ on 
$X$ such that every map $\varphi_x\colon X\to X$, $y\mapsto x\triangleright y$, is bijective, and
$x\triangleright(y\triangleright z)=(x\triangleright y)\triangleright(x\triangleright z)$ holds
for all $x,y,z\in X$. A rack is said to be \emph{indecomposable} if the group
generated by $\{\varphi_x:x\in X\}$ acts transitively on $X$. 

A rack is a \emph{quandle} if $x\triangleright x=x$ for all $x\in X$. 


In this work, we deal with a particular family of racks. 
An \emph{affine (or Alexander) rack} is a triple $(A,g,\triangleright)$, 
where $A$ is an abelian group, $g\in\Aut(A)$ and
$(a,b)\mapsto a\triangleright b$ is the binary operation
on $A$ given by 
$a\triangleright b=(\id-g)(a)+g(b)$ for all $a,b\in A$. 
In this case, we denote this rack by
$\Aff(A,g)$. 

An affine 
rack $\Aff(A,g)$ is \emph{indecomposable}
if and only if $\id-g$ is surjective. 

In this work, the following family 
of affine racks will be crucial. 
Let $p$ be a prime number
and $\Z/p\Z$ be the ring of integers modulo $p$. For 
$\alpha\in\Z/p\Z\setminus\{0,1\}$, 
$\Aff(p,\alpha)$ denotes the affine rack
$\Aff(\Z/p\Z,g)$, where 
\[
g\colon\Z/p\Z\to\Z/p\Z,
\quad 
g(x)=\alpha x.
\]
Then $\Aff(p,\alpha)$ is an affine indecomposable rack. 


The following lemmas give some information on the structure of groups generated by a conjugacy class. We will need this when studying Yetter--Drinfeld modules over such groups. For any element $g$ of a group $G$, we write $\langle g\rangle $ for the subgroup generated by $g$.

\begin{lem} 
\label{lem:GXmonomialsordered}
    Let $G$ be a group generated by a finite conjugacy class $X$,  let $m$ be the order of the conjugation action of any $x\in X$ on $X$, let $<$ be a total order of $X$, and let $z$ be the maximal element of $(X,<)$. Then for all $n\ge 0$, every element of $X^n\subseteq G$ is of the form $x_1^{n_1}x_2^{n_2}\cdots x_k^{n_k}z^l$ with $k,l\ge 0$, $n_1+n_2+\cdots +n_k+l=n$, $x_1<x_2<\cdots <x_k<z$, and $1\le n_i\le m-1$ for all $1\le i\le k$.  
\end{lem}

See Dietzmann's theorem in \cite[Theorem 5.10]{MR2426855} and 
Lemmas 2.18 and 2.19 in \cite{MR2803792} for related claims.

\begin{proof}
  The claim is trivial for $n\le 1$. Let $n\in \ndN$ with $n\ge 2$. Then in any $x_1x_2\cdots x_n\in X^n$, a factor $x_ix_{i+1}\in X^2$ with $1\le i\le k-1$ and $x_{i+1}<x_i$ can be replaced by the lexicographically smaller factor $x_{i+1}(x_{i+1}^{-1}x_ix_{i+1}) \in X^2$. We now prove that $x^m=z^m$ for all $x\in X$. Then using that $z^m$ is in the center of $G$, we can achieve that the exponents of the factors $x_i\ne z$ are at most $m-1$.
  
  Note that $x^my=yx^m=(yxy^{-1})^my$ for all $x,y\in X$, Hence $x^m=(yxy^{-1})^m$ for all $x,y\in X$. It follows that $x^m=(gxg^{-1})^m$ for all $x\in X$ and $g\in G$, since $G$ is generated by $X$. Moreover, $X$ is a conjugacy class of $G$, and hence $x^m=z^m$ for all $x\in X$. This completes the proof of the lemma.
\end{proof}

%
%

Our main objects of interest are Yetter--Drinfeld modules of 
prime dimension over non-abelian groups.

\begin{lem} \label{lem:Gcommfinite}
    Let $G$ be a group generated by a finite conjugacy class $X$. Then the derived subgroup $[G,G]$ of $G$ is finite and is generated by the elements $xy^{-1}$ with $x,y\in X$. Moreover, for all $z\in X$,
    \[ G=[G,G]\langle z\rangle\quad \text{and} \quad
    C_G(z)=\big([G,G]\cap C_G(z)\big)\langle z\rangle .
    \]
\end{lem}

\begin{proof}
  Since $G$ is generated by the conjugacy class $X$, the group $[G,G]$ is generated by the elements $uyu^{-1}y^{-1}$ with $u,y\in X$, and since $uyu^{-1}\in X$, by the elements $xy^{-1}$ with $x,y\in X$.

  Since $(xy^{-1})^{-1}=yx^{-1}$ and $y^{-1}x=(y^{-1}xy)y^{-1}$ for all $x,y\in X$, 
  \[ [G,G]=\{ x_1x_2\cdots x_n(y_1y_2\cdots y_n)^{-1} : 
  n\ge 0, x_1,x_2,\dots,x_n,y_1,y_2,\dots,y_n\in X\}. \]
  
  Let $z\in X$. Since
  \[ x=xz^{-1}z\in [G,G]z \quad \text{and} \quad x^{-1}=x^{-1}zz^{-1}\in [G,G]z^{-1}
  \]
  for all $x\in X$, it follows that $G=[G,G]\langle z\rangle $. Hence $C_G(z)=\big([G,G]\cap C_G(z)\big)\langle z\rangle $.

  It remains to prove that $[G,G]$ is finite.
  Let $<$ be a total order of $X$ with maximal element $z$. By Lemma~\ref{lem:GXmonomialsordered}, in this description, we may assume that the monomials $x_1x_2\cdots x_n$ and $y_1y_2\cdots y_n$ are ordered and do not contain factors of the form $x^m$, $x\in X\setminus \{z\}$, where $m$ is the order of the conjugation action of any $x\in X$ on $X$. Thus, $[G,G]$ is the set of all elements $\bar{x}z^l\bar{y}^{-1}$, where $\bar{x},\bar{y}$ are ordered monomials in the letters $X\setminus \{z\}$ without factors $x^m$, $x\in X\setminus \{z\}$, and $l=|y|-|x|$. This set is finite.
\end{proof}

\begin{lem}
\label{lem:GXformulas}
    Let $G$ be a group and let $X$ be a conjugacy class of $G$. Assume that $X$ generates $G$, and that there is a rack isomorphism $\varphi :\Aff(p,\alpha)\to X$ for some prime $p$ and some $\alpha \in \ndZ/p\ndZ \setminus \{0,1\}$. For all $i\in \ndZ/p\ndZ$ let $g_i=\varphi(i)$, and let $\gamma =g_0g_1^{-1}$.


    \begin{enumerate}
       \item The following relations hold in $G$:
    \begin{align*}
      g_xg_y^{-1} &=g_{x+z}g_{y+z}^{-1},\\
      g_0g_k^{-1}&=\gamma ^k,\\
      g_x\gamma &=\gamma ^\alpha g_x
    \end{align*}
    for all $x,y,z\in \ndZ/p\ndZ $ and $k\geq0$.
   \item The derived subgroup of $G$ is cyclic of order $p$ and is generated by $\gamma $.
  \end{enumerate}
\end{lem}

\begin{proof}
We write $\triangleright $ for the rack action of $X$.

(1) 
Since $\varphi $ is a rack isomorphism, we obtain that
\begin{align*}
 g_i\triangleright g_j&=g_{(1-\alpha)i+\alpha j}=g_{j+(1-\alpha)(i-j)},\\
 g_i^{-1}\triangleright g_j &=g_{j+(1-\alpha ^{-1})(i-j)}
\end{align*}
for all $i,j\in \ndZ/p\ndZ$.
Therefore, for all $i,j\in \ndZ/p\ndZ $ the relations
    \[ g_ig_j^{-1}=g_{i\triangleright j}^{-1}g_i
    =g_{i+(1-\alpha^{-1})((i\triangleright j)-i)}g_{i\triangleright j}^{-1}
    =g_{i+(1-\alpha )(i-j)}g_{j+(1-\alpha)(i-j)}^{-1} \]
    hold in $G$.
    Since the difference of $i+(1-\alpha)(i-j)$ and $j+(1-\alpha)(i-j)$ is
    $i-j$, it follows that 
    \[ 
    g_ig_j^{-1}=g_{i+k(\alpha -1)(i-j)}g_{j+k(\alpha-1)(i-j)}^{-1} 
    \]
    for all $k\in \mathbb{Z}$. Since $\alpha\ne 1$, we conclude that 
    $g_xg_y^{-1}=g_{x+z}g_{y+z}^{-1}$ for all $x,y,z\in \ndZ/p\ndZ$.
    As a direct consequence,
    \[ (g_0g_1^{-1})^k=g_0g_1^{-1}g_1g_2^{-1}\cdots g_{k-1}g_k^{-1} =g_0g_k^{-1}
    \]
    for all $k\ge 0$. Moreover,
    for all $x\in \ndZ/p\ndZ$ we obtain that
    \[ g_x\gamma =g_xg_0g_1^{-1}
    =g_{x\triangleright 0}g_{x\triangleright 1}^{-1}g_x=g_{(1-\alpha)x}g_{(1-\alpha)x+\alpha }^{-1}g_x=g_0g_\alpha^{-1}g_x=\gamma^\alpha g_x.\]
    
(2) By Lemma~\ref{lem:Gcommfinite}, the derived subgroup of $G$ is generated by the elements $g_xg_y^{-1}$ with $x,y\in X$. Thus $[G,G]$ is generated by $\gamma $ because of (1). Since $\gamma ^p=1$ and since $g_0\ne g_1$ in $G$, the order of $\gamma$ in $G$ is $p$.
\end{proof}

\begin{thm}[Etingof--Guralnick--Soloviev]
\label{thm:EGS}
    Any indecomposable quandle with a prime number
    of elements is isomorphic to $\Aff(p,\alpha)$ 
    for some prime number $p$ and $\alpha\in\Z/p\Z\setminus\{0,1\}$. 
\end{thm}

\begin{proof}
It follows from \cite[Theorems 2.5 and 3.1]{MR1848966}.
\end{proof}

Recall that a Yetter--Drinfeld module $V$ over a group $G$ is absolutely irreducible if and only if its support is a conjugacy class of $G$ and for some (equivalently, any) $x\in G$ with $V_x\ne 0$ the $\fie C_G(x)$-module $V_g$ is absolutely irreducible.

\begin{cor} \label{cor:Vp}
  Let $\fie $ be a field, $G$ be a group, and $V$ be an absolutely irreducible Yetter--Drinfeld module over $\fie G$ of dimension $p$ for some prime $p$. Let
  \[ X=\{ x\in G: V_x\ne 0\}, \]
  and assume that $G$ is generated by $X$.
  Then $X\cong \Aff(p,\alpha )$ for some $\alpha \in \ndZ/p\ndZ$, $\dim V_x=1$ for all $x\in X$, and there is a scalar $\lambda \in \fie ^\times $ and a family $(v_x)_{x\in X}$ of vectors $v_x\in V_x$, such that
  \[ x v_y=\lambda v_{x\triangleright y} \]
  for all $x,y\in X$.
\end{cor}

\begin{proof}
        Let $X$ be the support of $V$. Since $V$ is irreducible, $X$ is a conjugacy class of $G$. Since $\dim V=p$, we obtain that $|X|\in \{1,p\}$. By assumption, $X$ generates $G$. Since $V$ is absolutely irreducible and $\dim (V)=p>1$, the group $G$ can not be abelian and hence $|X|=p$. Thus $X$ is an indecomposable quandle of size $p$, and 
    it follows from Theorem~\ref{thm:EGS} that there is a rack isomorphism
    \[ \varphi : \Aff(p,\alpha )\to X \]
    for some 
    $\alpha \in \ndZ/p\ndZ \setminus \{0,1\}$.
    For all $i\in \ndZ/p\ndZ$ let $g_i=\varphi(i)\in X$, and let $\gamma=g_0g_1^{-1}$.
    Since $|X|=\dim V$, it follows that $\dim V_x=1$ for all $x\in X$.
    Let $w_0\in V_{g_0}$ be a non-zero element, $v_{\varphi(0)}=w_{0}$, and let $\lambda \in \fie $ with $g_0w_0=\lambda w_0$. For all $1\le i\le p-1$ let
    \[ 
    w_i= \gamma ^{i\beta }w_0, 
    \quad v_{\varphi(i)}=w_i,
    \]
    where $\beta\in \Z /p\Z $ with $\beta (\alpha -1)=1$. (Note that the claim of the corollary is on the family $(v_x)_{x\in X}$. We use the vectors $w_i$ in order to simplify the notation.)

    By Lemma~\ref{lem:GXformulas}(1),
    \[ g_iw_j
    =g_i\gamma^{j\beta }w_0
    =\gamma^{j\alpha \beta }g_ig_0^{-1}\lambda w_0
    =\gamma^{j\alpha \beta }g_0g_{p-i}^{-1}\lambda w_0=\gamma^{j\alpha \beta }\gamma ^{p-i}\lambda w_0.
    \]
    Now note that $p-i=(p-i)(\alpha -1)\beta $, and hence
    \[ 
    g_iw_j=\lambda\gamma ^{(j\alpha +i(1-\alpha ) )\beta } w_0=\lambda w_{i\triangleright j}. 
    \]
    This proves the corollary.
\end{proof}

The following lemma will be needed later for technical reasons.

\begin{lem} \label{lem:xr}
  Let $G$ be a group generated by a conjugacy class $x^G$ of an element $x\in G$. Let $m$ be the order of the conjugation action of $x$ on $x^G$. Then for all $r\in \ndN$ with $\gcd(r,m)=1$, $(x^r)^G$ is a conjugacy class of $G$ of size $|x^G|$.   
\end{lem}

\begin{proof}
  Let $a,b\in \ndZ$ with $ar+bm=1$.
  Then
  \[ (gx^rg^{-1})^a
  =gx^{ar}g^{-1}
  =gx^{1-bm}g^{-1}=gxg^{-1}x^{bm}
  \]
  for all $g\in G$, since $x^m$ is central in $G$.
  Hence the map $(x^r)^G\to x^G$, $y\mapsto y^ax^{-bm}$, is bijective.
\end{proof}

Lastly, for our analysis, the following proposition 
will be crucial. 

\begin{pro}
\label{pro:3+3}
    Let $\fie$ be a field. 
    Let $G$ be a group and 
    $V$ and $W$ be absolutely irreducible Yetter--Drinfeld modules
    over $\fie G$. Assume that the supports of $V$ and $W$ do not commute and
    their union generates $G$. If $\dim\NA(V\oplus W)<\infty$, 
    then 
    \[
    \{\dim V,\dim W\}=\{\{1,3\},\{1,4\},\{2\},\{2,3\},\{2,4\}\}.
    \]
\end{pro}

\begin{proof}
    This is a consequence of \cite[Theorem 2.1]{MR3656477} by  
    inspection of \cite[Table 1]{MR3656477}.
\end{proof}

At this point it is interesting to mention that Proposition~\ref{pro:3+3} is the main tool in \cite{MR3713037,MR4109132,MR4325965,MR4151588} 
to study Nichols algebras over simple groups. 

\section{Orders of braided vector spaces and Nichols algebras}
\label{ss:orders}

Assume that $\ch(\fie)=0$ and let $H$ be a Hopf algebra over $\fie $ with bijective antipode. To study Nichols algebras of Yetter--Drinfeld modules over $H$, we will use orders of the underlying braided vector spaces. 

\begin{defi} \label{de:YDorder}
Let $V$ be a finite-dimensional Yetter--Drinfeld module over $H$ and let
$R$ be a subring of $\fie $. An $R$-\emph{order of} the braided vector space $V$ is a finitely generated projective $R$-submodule $V_R$ of $V$ such that
\begin{enumerate}
  \item the canonical map $\fie \ot _R V_R\to V$, $\lambda \ot v\mapsto \lambda v$, is bijective, and
  \item
$c(V_R^{\ot 2})\subseteq V_R^{\ot 2}$ and $c^{-1}(V_R^{\ot 2})\subseteq V_R^{\ot 2}$.
\end{enumerate}
\end{defi}

\begin{exa}
  Let $R$ be a subring of $\fie $, which is a Dedekind domain, e.g., the ring of integers $\ndZ[q]$ of a cyclotomic field.  Then finitely generated $R$-submodules of vector spaces over $\fie $ are torsion-free and hence projective. In our applications, in particular in the proof of Theorem~\ref{thm:2primes} and of Theorem~\ref{thm:main}, the $R$-orders will be of this form.
\end{exa}

\begin{exa} Let $R$ be a subring of $\fie $ and let $V$ be a Yetter--Drinfeld module over $H$. Assume that $\dim(H)<\infty $ and let $X$ be a Hopf order of $H$. This means that $X$ is an $R$-order of $H$ as a $\fie$-vector space and that $X$ is closed under multiplication, comultiplication, unit, counit, and the antipode (see \cite{CM1} for precise definitions). 
The dual Hopf algebra $H^*$ admits a dual Hopf order $X^{\star}$. By definition, 
\[
X^{\star}= \{f\in H^*: f(X)\subseteq R\}.
\]
Recall that Yetter--Drinfeld modules can be considered as modules over the Drinfeld double $D(H)$ of $H$. As a coalgebra, $D(H)=(H^{op})^*\ot H$, and as an algebra the multiplication is given  by the rule 
\[
(f\ot a)(g\ot b) = g_{(1)}(S^{-1}(a_{(3)}))g_3(a_{(1)})fg_{(2)}\ot a_{(2)}b;
\]
see \cite[\S IX.4.1]{MR1321145} for more details. 
In particular, by using the formulas for the multiplication and comultiplication, it is straightforward to show that 
the image of $X^{\star}\ot_R X$ inside $D(H)$ provides a Hopf order of $D(H)$ that we will write as $D(X)$. The $R$-matrix is then given by $M=\sum_i e^i\ot e_i\in D(H)$, where $\{e_i\}$ is a basis of $H$ and $\{e_i\}$ denotes the dual basis of $H^*$. This element is necessarily contained in $D(X)$ as well. Indeed, $D(X)$ can be characterized as the set of all elements in $H^*\ot H$ such that their pairing with any element in $X\ot_R X^{\star}$ is in $R$, and it holds that $M(x\ot f ) = f(x)$.   

Let now $v_1,\dots,v_n$ be a basis of $V$. 
Then 
\[
\left\{ \sum_{i=1}^n t_iv_i: t_i\in D(X)\right\}\subseteq V
\]
is an example of an $R$-order of $V$ as a braided vector space.
\end{exa}

Recall 
from \cite[Definition 7.1.13]{MR4164719} 
that for any $V\in\yd{H}$, 
the Nichols algebra of $V$ is
\[ \NA (V)=T(V)/I_V, \]
where $T(V)$ is the tensor algebra of $V$ and
\[ I_V=\bigoplus_{n=2}^\infty \ker (\Delta_{1^n}:T^n(V)\to T^n(V))\subseteq
T(V),
\]
and for each $n\ge 2$, $\Delta_{1^n}$ is the $n$-th symmetrizer morphism.
It is well-known that $I_V$ is an ideal and coideal of $T(V)$.
The Nichols algebra $\NA (V)$ is an $\ndN_0$-graded algebra and coalgebra with homogeneous components
\[ \NA^n (V)=T^n(V)/\ker (\Delta_{1^n}) \]
for all $n\ge 0$. 

Now assume that $V$ is finite-dimensional. Let $R$ be a subring of $\fie $ and let
$V_R$ be an $R$-order of the braided vector space $V$. Then for all $n\in \ndN_0$, 
\[
V_R^{\ot n}:= V_R\ot_R V_R\ot_R\cdots\ot_R V_R
\]
is a finitely generated projective module. Moreover, the inclusion $V_R\to V$ induces
injections $V_R^{\ot n}\to V^{\ot n}$, and isomorphisms
\[ \fie \ot_R V_R^{\ot n}\stackrel{\cong}{\to} V^{\ot n}. \]
Hence
\[ \NA (V_R)=\bigoplus _{n=0}^\infty
\big(V_R^{\ot n}
+\ker(\Delta_{1^n})\big)
/\ker (\Delta_{1^n})
\cong \bigoplus_{n=0}^\infty V_R^{\ot n}
/\big(V_R^{\ot n}\cap \ker (\Delta_{1^n})\big) \]
is a graded braided Hopf algebra over $R$. As an algebra, it is generated in degree one. We call it the \emph{Nichols algebra of $V_R$}. We would like to stress that this Hopf algebra is
one of the basic points of the paper. 

\begin{rem}
    In the literature, authors use different approaches to introduce orders. Alternatively to our definition,
    we could start with an integral domain $R$, a Hopf order $H_R$ over $R$, and a Yetter--Drinfeld module $V_R$ over $H_R$ which is finitely generated projective as an $R$-module. Then $\NA(V_R)$ could be introduced as the quotient of the tensor algebra of $V_R$ by the kernel of $\oplus_n\Delta_{1^n}$. Note however, that our definition of an $R$-order
    of $V$ does not require the existence of an $R$-order of $H$. 
\end{rem}

Let now $\fm $ be a maximal ideal of $R$. Then $R/\fm $ is a field, and for each finitely generated projective $R$-module $M$,
$R/\fm \ot_R M$ is a vector space over $R/\fm $ of dimension $\dim _\fie (\fie \ot_R M)$. (This follows e.g.~from Theorem~1 of \cite[V.2]{MR1727221}; see also the remark on page 111.)
Let $c=c_{V,V}$ be the braiding of $V$.
By Definition~\ref{de:YDorder},
\[ \id \ot c \in \End_{R/\fm}(R/\fm \ot_R V_R^{\ot 2}). \]
It follows that
\[ V_{R,\fm}:=R/\fm \ot_R V_R \]
is a braided vector space
and
$R/\fm \ot_R \NA (V_R)$
is an $\ndN_0$-graded braided Hopf algebra
with degree one part $V_{R,\fm}$. As an algebra, $R/\fm \ot_R \NA (V_R)$ is generated by $V_{R,\fm}$. Hence $R/\fm \ot _R \NA (V_R)$ is a pre-Nichols algebra of $V_{R,\fm}$.
In other words, we have a canonical surjection $$R/\fm\ot_R \NA(V_R)\to \NA(V_{R,\fm}).$$ 
The analysis in this section is based on the observation that the above surjection might not be injective, which means that it is possible that $R/\fm\otimes_R\NA (V_{R})$ will not be a Nichols algebra. This has consequences for the structure of $\NA (V)$.

We keep our notation and assumptions regarding $\fie $, $H$, $V$ and $R$.
The following lemma is immediately clear from the preceding discussion.

\begin{lem} \label{lem:specialfinite}
Let $V_R$ be an $R$-order of $V$ and let $\fm $ be a maximal ideal of $R$. Then $c_{V,V}$ induces a braided vector space structure on $V_{R,\fm}$, and $R/\fm \ot _R
\NA (V_R)$ is a pre-Nichols algebra of $V_{R,\fm}$. If
$\NA (V)$ is finite-dimensional, then $R/\fm \ot _R \NA (V_R)$ and $\NA (V_{R,\fm})$ are finite-dimensional.
\end{lem}

We recall Takeuchi's notion of categorical subspaces, see for example 
\cite[Definition 6.1.5]{MR4164719}. A \emph{categorical subspace} of a braided vector space $(W,c)$ 
is a subspace $U$ of $W$ such that 
\[
c(U\ot W)=W\ot U\text{ and } 
c(W\ot U)=U\ot W.
\]

Next we formulate a very useful criterion to identify infinite-dimensional Nichols algebras. Another one we will discuss in Section~\ref{ss:def_bvs}.

\begin{lem} 
\label{lem:specialprimitives}
Let $V_R$ be an $R$-order of $V$ and let $\fm$ be a maximal ideal of $R$. Let $W\subseteq \bigoplus_{n\ge 2}R/\fm \ot_R \NA^n (V_R)$
be a categorical subspace consisting of primitive elements. Assume that the Nichols algebra of $V_{R,\fm}\oplus W$ is infinite-dimensional. Then $\NA (V)$ is infinite-dimensional.
\end{lem}

\begin{proof}
  Let $\cF$ be the increasing algebra filtration by $\ndN_0$ of $R/\fm \ot _R \NA (V_R)$ such that $V_{R,\fm}$ and $W$ are in degree one. Then $\cF$ is a Hopf algebra filtration. To prove that $\cF$ is a coalgebra filtration it is needed that $W$ is a categorical subspace. The graded Hopf algebra associated to $\cF$ is a pre-Nichols algebra of 
  $V_{R,\fm}\oplus W$ by construction. Since its canonical quotient Nichols algebra is infinite-dimensional by assumption, it follows that
  $R/\fm \ot_R \NA (V_R)$ is infinite-dimensional, and hence $\NA (V)$ is infinite-dimensional by Lemma~\ref{lem:specialfinite}.
\end{proof}

\begin{thm} \label{thm:2primes}
    Assume that $\ch(\fie)=0$. Let $V$ be an absolutely irreducible Yetter--Drinfeld module over a group $G$
    such that $\dim V_x\le 1$ for all $x\in G$. 
    Assume that
    the support $X=\{x\in G:V_x\ne 0\}$ of $V$ generates $G$ and
    has at least three elements. Let $m$ be the order of the conjugation action of any $x\in X$ on $X$.
    Let $\lambda\in\fie\setminus\{0\}$ be such that
    \[
    xv=\lambda v\qquad 
    \text{for all $x\in X$ and $v\in V_x$},
    \]
    and that
    \begin{enumerate}
        \item $\lambda $ is a root of $1$,
        \item the order $N$ of $\lambda$ is divisible by at least two distinct prime factors, and
        \item either $\gcd(m,N)=1$ or there is a (unique) prime number $p$ such that $\gcd(m,N)=p^k$ for some $k\geq1$. 
    \end{enumerate}  
    Then $\dim\NA(V)=\infty$. 
\end{thm}

\begin{proof} 
    Let $z\in X$ and let $0\ne v_z\in V_z$. Let $R$ be the smallest subring of $\fie $ such that
    \begin{align} \label{eq:gvz}
        gv_z\in Rv_z\quad \text{for all $g\in C_G(z)$.}
    \end{align}
    Note that 
   $C_G(z)=\big( [G,G]\cap C_G(z)\big)\langle z\rangle$.
    Since $[G,G]$ is finite by Lemma~\ref{lem:Gcommfinite} and $zv_z=\lambda v_z$, where $\lambda $ is a root of $1$, we conclude that $R$ is an extension of $\ndZ$ by a root of $1$, and $\lambda \in R$.
    
    Now we define an $R$-order of $V$.
    For all $y\in X$ let $h_y\in G$ be such that $h_yv_z\in V_y$, where $h_z=1$. Let $V_R$ be the (free) $R$-submodule of $V$ generated by $h_yv_z$ for all $y\in X$. Then
    Equation~\eqref{eq:gvz} implies that
    $V_R$ is an $R$-order of $V$.

    If $\gcd(m,N)=1$, let $p$ be a prime divisor of $N$. Note that otherwise $p$ is by assumption the unique common prime divisor of $m$ and $N$.
    
    Let $l\ge 1$ and $r\ge 2$ be such that $N=p^lr$ and $\gcd(p,r)=1$. Let \begin{align} \label{eq:WR}
       W_R=\sum_{y\in X} R(h_yv_z)^r \subseteq \NA (V_R). 
    \end{align}
    Note that $W_R$ is an $RG$-subcomodule and an $RG$-submodule, since the action of $G$ on $V_R$ permutes the $R$-submodules $Rh_yv_z$ with $y\in X$.
    Moreover, 
    \begin{align} \label{eq:Deltahyvz}
        \Delta( (h_yv_z)^r)=
        \sum_{i=0}^r \binom ri_\lambda (h_yv_z)^i\ot (h_yv_z)^{r-i}
    \end{align}
    for all $y\in X$, where the scalars $\binom{r}{i}_\lambda$ are 
    Gaussian binomial coefficients (see e.g. \cite[Section 1.9]{MR4164719}).

    Let $\fm$ be a maximal ideal of $R$ containing $p$. Then $W_{R,\fm}=R/\fm\ot _R W_R$ is a Yetter--Drinfeld submodule over $(R/\fm)G$ of $R/\fm \ot _R \NA (V_R)$. Since the order of $\lambda $ in $R/\fm$ is $r$, Equation~\eqref{eq:Deltahyvz} implies that $W_{R,\fm}$ consists of primitive elements.

Next we prove that $W_{R,\fm}$ is absolutely irreducible as a Yetter--Drinfeld module over $(R/\fm )G$ and determine its dimension. Equation~\eqref{eq:WR} implies that the support of $W_{R,\fm}$ is $\{x ^r:x\in X\}$. Since $\gcd(r,m)=1$, Lemma~\ref{lem:xr} implies that the support of $W_{R,\fm}$ consists of $|X|$ elements. Then $\dim (W_{R,\fm})_{z^r}=1$, and $W_{R,\fm}$ is absolutely irreducible. Note that $|X|\ge 3$.
Therefore $\dim \NA (V_{R,\fm}\oplus W_{R,\fm})=\infty $ by Proposition~\ref{pro:3+3} and hence $\dim \NA (V)=\infty $ by Lemma~\ref{lem:specialprimitives}.
\end{proof}

\section{Nichols algebras in positive characteristic}
\label{ss:positive}

%
%


%
%
%

\subsection{Nichols algebras of diagonal type}

Let $\cD=\cD(G,(g_i)_{i\in \{1,2\}},(\chi_i)_{i\in \{1,2\}})$ be a Yetter--Drinfeld datum in the 
sense of \cite[Definition 8.2.2]{MR4164719}. 
Let 
\[ (q_{ij})_{i,j\in \{1,2\}}=(\chi_j(g_i))_{i,j\in \{1,2\}}
\]
be the braiding matrix of $\cD$.
Let $V\in \yd{\fie G}$ be a Yetter--Drinfeld module defined by $\cD $ with basis $x_1,x_2$, see \cite[Section~8.3]{MR4164719}. 
Thus
\[ \delta(x_i)=g_i\ot x_i,\quad
gx_i=\chi_i(g)x_i \]
for all $i\in \{1,2\}$ and $g\in G$. Then $V$ is a braided vector space of diagonal type with braiding $c=c_{V,V}$ and
\[ c(x_i\ot x_j)=q_{ij}x_j\ot x_i \]
for all $i,j\in \{1,2\}$.

The generalized Dynkin diagram of the braiding matrix $(q_{ij})_{i,j\in \{1,2\}}$ is the labeled graph
    \begin{center}
    \Dchaintwo{$q_{11}$}{$r$}{$q_{22}$}
    \end{center}
    with $r=q_{12}q_{21}$,
where the edge and the label $q_{12}q_{21}$ are omitted if $q_{12}q_{21}=1$.

The following result is a particular case
of the classification 
of rank two Nichols algebras
of diagonal type in positive characteristic: 

\begin{pro}
\label{pro:rank2}
    Assume that $\ch(\fie)=p>0$. 
    Let $V$ be a two-dimensional 
    braided vector space of diagonal type with 
    generalized Dynkin diagram
    \[
    \Dchaintwo{1}{a}{b}
    \]
    where $a\in\{2,3,\dots,p-1\}$ and $0\ne b\in\fie$. 
    Then $\dim\NA(V)<\infty$ if and only if
    \[
    (p,a,b)\in\{(3,2,-1),(5,2,-1),(5,3,-1),(7,3,-1),(7,5,-1)\}.
    \]
\end{pro}

\begin{proof}
    See \cite[Theorem 5.1 and Remark 5.3]{MR3313687}. 
\end{proof}

For all $m\geq0$, let 
$\beta_m=(\ad x_1)^m(x_2)$ in $\NA(V)$. 
By \cite[Proposition 4.3.12]{MR4164719}, 
\begin{align*}
\Delta(\beta_m)&=\beta_m\ot 1+\sum_{k=0}^m \binom{m}{k}_{q_{11}}\left(\prod_{l=m-k}^{m-1}(1-q_{11}^lq_{12}q_{21})\right)x_1^{k}g_1^{m-k}g_2\ot \beta_{m-k}
\end{align*}
in $\NA(V)\#\fie G$ for all $m\ge 0$. 

Assume now that
$q_{11}=q_{12}=1$ and let $a=q_{21}$ and $b=q_{22}$. Thus the generalized Dynkin diagram
of $V$ is
\begin{center}
    \Dchaintwo{$1$}{$a$}{$b$}
\end{center}
Moreover, 
\[
\Delta(\beta_m)=\beta_m\otimes 1+\sum_{k=0}^m \binom{m}{k}(1-a)^{k}x_1^{k}
g_1^{m-k}g_2\ot\beta_{m-k}
\]
in $\NA(V)\#\fie G$ for all $m\ge 0$. 

Let $\cA=\NA (\fie x_1)\#\fie G$ and let $K=(\NA (V)\#\fie G)^{\mathrm{co}\,\cA }$, see \cite[Section~13.2]{MR4164719}. Let
\begin{align} \label{eq:W}
    W=(\ad \cA )(\fie x_2)\subseteq \NA (V).
\end{align}
Then $W\in \yd{\cA}$ by \cite[Proposition~13.2.4]{MR4164719}, and consists of primitive elements of the Hopf algebra $K\in \yd{\cA}$.
The vector space
$W$ is spanned by the elements $\beta_m$ with $m\ge 0$.
The braiding of $W$ is given by
\begin{equation}
\label{eq:braiding}
c(\beta_m\otimes \beta_n)=b\sum_{k=0}^m \binom{m}{k}(1-a)^k a^n
\beta_{n+k}\otimes \beta_{m-k}.
\end{equation}
By \cite[Theorem~13.2.8]{MR4164719},
$K\cong \NA (W)$ as braided Hopf algebras in $\yd{\cA}$. Moreover,
\[ \NA (V)\# \fie G \cong K\# \cA\quad \text{and} \quad
\NA (V)\cong K\# \fie [x_1]
\] 
via bosonization.

Let now $p$ be a prime number and assume that $\ch (\fie )=p$. Then $x_1^p=0$ in $\NA (V)$ (since $q_{11}=1$) and $\beta_p$ is $g_1^pg_2$-primitive in $\NA (V)\# \fie G$ and hence $\beta_p=0$ in $\NA (V)$. In particular, 
$\dim K<\infty$ if and only if $\dim\NA(V)<\infty$. Therefore, 
by Proposition~\ref{pro:rank2}, the following holds.

\begin{lem} \label{lem:Kfinite}
  Let $W,K,p,a,b$ as above. Assume that $\ch(\fie )=p$. Then the Nichols algebra $\NA (W)\cong K$ is finite-dimensional if and only if
\[
(p,a,b)\in\{(3,-1,-1),(5,2,-1),(5,3,-1),(7,3,-1),(7,5,-1)\}.
\]
\end{lem}

\subsection{Deformation of the braided vector space}
\label{ss:def_bvs}

\begin{defi}
  Let $(V,c)$ be a braided vector space.
    A \emph{decreasing filtration of} $V$ is a family $(\cF^iV)_{i\ge 0}$ of subspaces $\cF^iV\subseteq V$ such that
    \begin{enumerate}
        \item 
    $V=\cF^0V\supseteq \cF^{k}V\supseteq \cF^{l}V$ for all $0\le k\le l$,
    \item $\bigcap_{i\ge 0}\cF^iV=\{0\}$, and
    \item 
    $c(\cF^iV\ot \cF^{j}V)\subseteq\bigoplus_{k+l\geq i+j} \cF^{k}V\ot \cF^{l}V$ for all $i,j\ge 0$.
    \end{enumerate}
\end{defi}

\begin{defi} 
\label{de:assgrbrvec}
Let $(V,c)$ be a braided vector space with a decreasing filtration $(\cF^iV)_{i\ge 0}$. 
The pair $(V^{\gr},c^{\gr })$ with
\begin{align*}
  V^{\gr}&=\bigoplus_{i\ge 0}\cF^iV/\cF^{i+1}V,\\
  c^{\gr}:&V^\gr \ot V^\gr \to V^\gr \ot V^\gr,&& x\ot y\mapsto c(x\ot y)+
  \sum_{k+l>i+j}\cF^{k}V\ot \cF^{l}V\\
  &&&
  \text{for all $x\in \cF^iV$, $y\in \cF^{j}V$,
  $i,j\ge 0$,}
\end{align*}
is called the \emph{associated graded braided vector space}. 
We also say that $(V,c)$ \emph{degenerates to} $(V^\gr,c^\gr)$.
\end{defi}

\begin{rem} 
In \cite{MR4401961}, the terminology ``specializes'' instead of ``degenerates'' is used.
\end{rem}

\begin{rem}
Let $V$ be an object in a braided monoidal category, where $V$ is in particular a vector space. In Definition~\ref{de:assgrbrvec}, we do not require that the filtration consists of objects in the same category. In particular, the degeneration is possibly not an object of the category where $V$ comes from.
\end{rem}

The notion introduced in Definition~\ref{de:assgrbrvec} is 
justified by the following lemma. The proof of this lemma is fairly elementary and is omitted.

\begin{lem}
Let $(V,c)$ be a braided vector space and $(\cF^iV)_{i\ge 0}$ be a decreasing filtration of $V$. 
Then $(V^\gr,c^\gr)$ is a braided vector space, and
\[ 
c^\gr(\cF^iV/\cF^{i+1}V\ot \cF^{j}V/\cF^{j+1}V)
\subseteq \bigoplus_{k+l=i+j}
\cF^{k}V/\cF^{k+1}V\ot \cF^{l}V/\cF^{l+1}V. 
\]
\end{lem}

For our analysis, the following observation will be crucial.

\begin{pro} \label{pr:Vgr}
Let $(V,c)$ be a braided vector space with a decreasing filtration $(\cF^iV)_{i\ge 0}$. 
Then $\gr \, \NA (V)$ is a pre-Nichols algebra of $V^\gr$. In particular, if $\NA (V)$ is finite-dimensional, then $\NA (V^\gr)$ is finite-dimensional.   
\end{pro}

For Yetter--Drinfeld modules, there is an elementary construction of decreasing filtrations using decreasing Hopf algebra filtrations of finite length.

For any Hopf algebra $H$ with comultiplication $\Delta$, counit $\epsilon$ and antipode $S$, we say that a family
$(H_i)_{i\ge 0}$ of subspaces of $H$ is a \emph{decreasing Hopf algebra filtration}~if
the following conditions hold: 
\begin{enumerate}
    \item $H_0=H$, and $H_i\supseteq H_j$ for all $0\le i\le j$.
    \item $H_iH_j\subseteq H_{i+j}$ for all $i,j\ge 0$.
    \item $\Delta(H_i)\subseteq \sum_{j=0}^i H_j\ot H_{i-j}$ for all $i\ge 0$.
    \item $\varepsilon (H_i)=0$ and $S(H_i)\subseteq H_i$ for all $i\ge 1$.
\end{enumerate}
We say that this filtration has \emph{finite length}, if $H_n=\{0\}$ for some $n\ge 1$.


\begin{pro}
\label{pr:decfiltbyHopffilt}
    Let $H$ be a Hopf algebra with bijective antipode and  $(H_i)_{0\le i\le n}$ be a decreasing Hopf algebra filtration of $H$ with $H_n=\{0\}$.
Let $V\in \yd{H}$, and for all $i\ge 0$ let $\cF^iV=H_iV$. Then
    \[ 
    \delta (\cF^iV)\subseteq \sum_{k\ge 0}H_k \ot \cF^{i-k}V 
    \]
    for all $i\ge 0$,
    and $(\cF^iV)_{i\ge 0}$ is a decreasing filtration of the braided vector space $V$.
\end{pro}

\begin{proof}
  It is clear that $\cF^{0}V=V$ and $\cF^{i+1}V=H_{i+1}V\subseteq H_iV=\cF^iV$ for all $i\ge 0$. Moreover, $\cF^{n}V=H_nV=\{0\}$.
  
  We now prove the claim about $\delta (\cF^iV)$ for all $i\ge 0$.
  Let $i\ge 0$, $v\in V$ and $h\in H_i$. Then
  \begin{align*}
    \delta(hv)
&=h_{(1)}v_{(-1)}S(h_{(3)})\ot h_{(2)}v_{(0)}\\
  &
  \in \sum_{j+k+l=i}
  H_jH_0S(H_l)\ot H_kV\subseteq
  \sum_{k\ge 0} H_{i-k}\ot H_kV
  \end{align*}
  since $(H_i)_{0\le i\le n}$ is a decreasing Hopf algebra filtration of $H$ and since $\cF^iV=H_iV$.
  
  Finally, for all $i,j\ge 0$ we obtain that
  \[ c(\cF^iV\ot \cF^{j}V)\subseteq \sum_{k\ge 0}H_k\cF^{j}V\ot \cF^{i-k}V=\sum_{k\ge 0} \cF^{j+k}V\ot \cF^{i-k}V,
  \]
  and hence $(\cF^iV)_{i\ge 0}$ is a decreasing filtration of $V$.
\end{proof}

\begin{cor} \label{co:decfiltJ}
  Let $H$ be a Hopf algebra with bijective antipode, and let $J\subseteq H$ be a nilpotent Hopf ideal and $H_i=J^i$ for all $i\ge 0$. Let $V\in \yd H$.
  \begin{enumerate}
  \item The family $(J^i)_{i\ge 0}$ is a decreasing Hopf algebra filtration of $H$ of finite length.
      \item 
The family
  $(J^iV)_{i\ge 0}$ is a decreasing filtration of the braided vector space $V$.
  \item Let
  \[ H^\gr=\bigoplus_{i\ge 0}H_i/H_{i+1}. \]
  Then $H^\gr$ is an $\ndN_0$-graded Hopf algebra.
  \item 
  The $H$-action and the $H$-coaction on $V$ induce an $H^\gr$-action and an $H^\gr$-coaction on
  \[ V^\gr=\bigoplus_{i\ge 0} J^iV/J^{i+1}V. \]
  With them, $V^\gr $ is an $\ndN_0$-graded Yetter--Drinfeld module over $H^\gr$.
  \end{enumerate}
\end{cor}

\begin{proof}
  (1) The defining properties of a decreasing Hopf algebra filtration follow from the fact that $J$ is a Hopf ideal of $H$. The finite length property follows from the nilpotency of $J$.
  
  (2) Apply Proposition~\ref{pr:decfiltbyHopffilt} with $H_i=J^i$ for all $i\ge 0$.

  (3) This is mainly due to (1). The claim can be proven by standard arguments.

  (4) The $H$-action and $H$-coaction on $V\in \yd H$ are filtered morphisms. This implies the claim.
\end{proof}

\begin{rem}
The family $(J^i)_{i\ge 0}$ is also known as the $J$-adic 
topology of $H$, and the corresponding family $(J^iV)_{i\ge 0}$ as the $J$-adic topology of $V$.
\end{rem}

\begin{pro}
\label{pro:J}
 Let $\fie $ be a field of characteristic $p>0$. Let 
 $G$ be a group and $V$ be an absolutely irreducible 
 Yetter--Drinfeld module over $\fie G$ of dimension $p$. 
 Assume that 
  \[ 
  X=\{ x\in G: V_x\ne 0\}
  \]
  generates $G$. Let $\varphi: \Aff(p,\alpha )\to X$ be a rack isomorphism with $\alpha \in \ndZ/p\ndZ\setminus \{0,1\}$. 
  Let $g_0=\varphi(0)$, $\gamma =\varphi(0)\varphi(1)^{-1}$, $0\ne v_0\in V_{g_0}$, and $\lambda \in \fie $ be such that $g_0v_0=\lambda v_0$.
  
\begin{enumerate}
    \item 
The ideal $J\subseteq \fie G$ generated by $\gamma-1$ is a nilpotent Hopf ideal of $\fie G$.
\item The $J$-adic filtration $(J^i)_{i\ge 0}$ of $\fie G$ defines a filtration $(J^iV)_{i\ge 0}$ of the braided vector space $V$.
\item For each $0\le j\le p-1$, $V^\gr (j)$ is spanned linearly by $(\gamma-1)^jv_0$. The structure maps of the Yetter--Drinfeld module $V^\gr$ are determined by
\begin{align*}
g_0 (\gamma-1)^jv_0&=\lambda \alpha^j(\gamma-1)^jv_0,\\
\delta( (\gamma-1)^jv_0)&=
\sum_{i=0}^j \binom j i (1-\alpha)^i(\gamma-1)^ig_0\ot (\gamma-1)^{j-i}v_0
\end{align*}
for all $0\le j\le p-1$.
\end{enumerate}
\end{pro}

\begin{proof}
    The isomorphism $\varphi$ exists by Corollary \ref{cor:Vp}. 
    For all $i\in \ndZ/p\ndZ$ let $g_i=\varphi(i)$.
    
    (1) It is clear that $J$ is a Hopf ideal. By Lemma~\ref{lem:GXformulas}(1), 
    \begin{align}
    \label{eq:gxg-1}
        g_i(\gamma-1)=(\gamma^\alpha -1)g_i\in (\gamma-1)\fie G,
    \end{align}
    and hence $G(\gamma-1)\subseteq (\gamma -1)\fie G$. Similarly, $(\gamma-1)G\subseteq \fie G(\gamma-1)$,
    and hence
    \begin{align} \label{eq:J} J=\fie G(\gamma -1)=(\gamma -1)\fie G.
    \end{align}
    Moreover, $(\gamma-1)^p=\gamma^p-1=0$. Therefore $J$ is nilpotent.

    (2) follows from Corollary~\ref{co:decfiltJ}.
    
    (3) Let $0\le j\le p-1$. By definition, $(\gamma -1)^jv_0\in J^jV$. Moreover, since $V=\fie G v_0$, Equation~\eqref{eq:J} implies that $J^jV=\fie G(\gamma-1)^jv_0$.

    By Equation~\eqref{eq:gxg-1},
    \[ 
    g_0(\gamma-1)+J^2=(\gamma^\alpha -1)g_0 +J^2=(\gamma-1)\sum_{i=0}^{\alpha -1}\gamma^ig_0 +J^2=\alpha (\gamma-1)g_0 +J^2. 
    \]
    Since $g_0v_0=\lambda v_0$, it follows by induction on $j$ that
    \[ 
    g_0(\gamma-1)^jv_0+J^{j+1}V=\alpha ^j\lambda (\gamma-1)^jv_0 +J^{j+1}V.  
    \]
    For all $1\le i\le p-1$,
    \[ 
    g_i=g_ig_0^{-1}g_0=g_0g_{p-i}^{-1}g_0=\gamma^{p-i}g_0.
    \]
    Hence
    \begin{align*}
    g_i(\gamma-1)^jv_0 &=\gamma^{p-i}g_0(\gamma-1)^jv_0\\
    &\in \fie \gamma^{p-i}(\gamma-1)^jv_0+J^{j+1}V
    \subseteq \fie (\gamma-1)^jv_0+J^{j+1}V. 
    \end{align*}
    Therefore 
    \[
    J^jV=\fie (\gamma -1)^jv_0+J^{j+1}V.
    \]

    The proof of the formula for $\delta( (\gamma-1)^jv_0)$ follows similarly by induction on $j$.
    \end{proof}

\begin{thm}
\label{thm:char_p}
Let $p$ be a prime number and 
assume that $\ch(\fie)=p$. 
Let $V$ be an absolutely irreducible Yetter--Drinfeld module
over a group $G$ such that $\dim V=p$ and the support of $V$ generates
$G$. Then
$\NA(V)$ is finite-dimensional if and only if
$p\in\{3,5,7\}$ and $V$ is isomorphic to one of the Yetter--Drinfeld modules
of Examples~\ref{exa:dim12}, \ref{exa:dim1280a}, \ref{exa:dim1280b},
~\ref{exa:dim326592a},~\ref{exa:dim326592b}. 
\end{thm}

\begin{proof}

Let $X$ be the support of $V$. By Corollary~\ref{cor:Vp}, $X\cong \Aff(p,\alpha )$ for some $\alpha \in \ndZ/p\ndZ \setminus \{0,1\}$, $\dim V_x=1$ for all $x\in X$, and there exist a scalar $\lambda \in \fie ^\times $ and a basis $(v_x)_{x\in X}$ of $V$ such that 
\[ 
xv_y=\lambda v_{x\triangleright y},\quad \delta(v_y)=y\ot v_y 
\]
for all $x,y\in X$. 

Assume first that 
    \[
    (p,\alpha,\lambda)\in\{(3,2,-1),(5,2,-1),(5,3,-1),(7,3,-1),(7,5,-1)\}.
    \]
    Then 
    $\NA(V)$ is one of the Nichols algebras
    of Examples~\ref{exa:dim12}, \ref{exa:dim1280a}, \ref{exa:dim1280b},
~\ref{exa:dim326592a},~\ref{exa:dim326592b} and therefore it is
finite-dimensional.

Assume now that 
    \[
    (p,\alpha,\lambda)\not\in\{(3,2,-1),(5,2,-1),(5,3,-1),(7,3,-1),(7,5,-1)\}.
    \]
    By assumption, $\ch(\fie)=p$. Let $(J^jV)_{j\ge 0}$ be the $J$-adic filtration of $V$ in Proposition~\ref{pro:J} and let $V^\gr$ be the associated graded Yetter--Drinfeld module. Then by functoriality, the filtration induces a filtration of $\NA (V)$. By Proposition~\ref{pr:Vgr}, $\gr\, \NA (V)$ is a pre-Nichols algebra of $\NA (V^\gr )$. It suffices to prove that $\dim \NA (V^\gr)=\infty $.

    By Proposition~\ref{pro:J}(3),
    there exist $g_0\in X$ and $0\ne v_0\in V_{g_0}$ such that the elements $y_j=(\gamma-1)^jv_0$ with $0\le j\le p-1$ form a basis of $V^\gr $. Moreover, the Yetter--Drinfeld structure of $V^\gr$ in Proposition~\ref{pro:J}(3) implies that
    \begin{align*}
    c^\gr(y_m\ot y_n)&=
    (y_m)_{(-1)} y_n\ot (y_m)_{(0)}\\
    &=\sum_{i=0}^m\binom m i (1-\alpha)^i(\gamma-1)^ig_0 y_n \ot y_{m-i}\\
    &=\sum_{i=0}^m\binom m i(1-\alpha )^i \lambda \alpha^n y_{n+i} \ot y_{m-i}
    \end{align*}
    for all $0\le m,n\le p-1$.
    Let $W$ be the braided vector space in Equation~\eqref{eq:W} corresponding to the parameters $a=\alpha $ and $b=\lambda $. By~\eqref{eq:braiding} and the above formula for $c^\gr$, the linear map
    \[ V^\gr \to W,\quad 
    y_m\mapsto \beta _m, \]
    is an isomorphism of braided vector spaces.
     Thus
     $\dim\NA(V^{\mathrm{gr}})=\infty$ by Lemma~\ref{lem:Kfinite}, and the proof is completed.
\end{proof}

\section{Proof of Theorem~\ref{thm:main}}
\label{ss:proof}

Let $X$ denote the support of $V$. By Corollary~\ref{cor:Vp}, 
$X\cong \Aff(p,\alpha )$ for $p=\dim (V)$ and some $\alpha \in \ndZ/p\ndZ\setminus\{0,1\}$, $\dim V_x=1$ for all $x\in X$, and there is a scalar $\lambda \in \fie ^\times $ and a family $(v_x)_{x\in X}$ of vectors $v_x\in V_x$, such that
\[ 
  x v_y=\lambda v_{x\triangleright y} 
\]
for all $x,y\in X$.

If $\lambda $ is not a root of $1$, or $\lambda=1$, then for all $x\in X$, $\NA (V_x)$ is infinite-dimensional (see e.g.~\cite[Example~1.9.6]{MR4164719}), and hence $\NA (V)$ is infinite-dimensional.

Assume that $\lambda $ is a root of $1$ and $\lambda \ne 1$. Let $R=\ndZ [\lambda ]\subseteq \fie $. Let $V_R$ be the (free) $R$-submodule of $V$ generated by the vectors $v_x$ with $x\in X$. Then $V_R$ is an $R$-order of the braided vector space $V$. Let $\fm$ be a maximal ideal of $R$ containing $p$. Then $\ch(R/\fm)=p$.

Assume first that $\lambda \ne -1$ in $R/\fm$ or
   \[
    (p,\alpha)\notin\{(3,2),(5,2),(5,3),(7,3),(7,5)\}.
    \]
    Then $\NA (V_{R,\fm})$ is infinite-dimensional by Theorem~\ref{thm:char_p}. Hence $\NA (V)$ is infinite-dimensional by Lemma~\ref{lem:specialfinite}.
    
Assume now that $\lambda = -1$ in $R/\fm$ and
   \[
    (p,\alpha)\in\{(3,2),(5,2),(5,3),(7,3),(7,5)\}.
    \]
If $\lambda =-1$ in $R$, then $\lambda=-1$ in $\fie $ and hence $V$ is one of the Examples~\ref{exa:dim12}--\ref{exa:dim326592b}. In this case, $\NA (V)$ is finite-dimensional.
Otherwise, assume that $\lambda \ne -1$ in $R$. Then the order of $\lambda $ in $\fie $ is $2p^k$ for some $k\ge 1$. 
Moreover, since 
$|\alpha|$ is the order of the conjugation action, 
$\gcd(|\alpha|,p)=1$. In particular, $\gcd(2p^k,|\alpha|)\in \{1,2\}$.
Then $\NA(V)$ is
infinite-dimensional by Theorem~\ref{thm:2primes}, and the proof of the theorem is completed.

\begin{rem}
    Theorem~\ref{thm:main} is not valid if $V$ is assumed to be irreducible but not absolutely irreducible. Indeed, let $\fie =\mathbb{R}$ and let $V=V_g$ be a $2$-dimensional Yetter--Drinfeld module over $G=\ndZ$, where $g$ is a generator of the group $\ndZ$. Assume that $g^2+g+1$ acts by $0$ on $V$. Then $\mathbb{C}\ot _{\mathbb{R}}V$ is a braided vector space of diagonal type, and $\dim \NA (V)=9$.
\end{rem}

\section{An example}
\label{ss:example}

In this section, we work out explicitly our results in the case
of 3-dimensional Yetter--Drinfeld modules.

Let $\fie$ be a field of characteristic zero.  
Let 
\[ 
G=\langle s_1,s_2:s_1^2=s_2^2,\,s_1s_2s_1=s_2s_1s_2\rangle.
\]
Then $G$ is a central
extension of the symmetric group $\mathbb{S}_3$. Let $X$ be the conjugacy class of 
$s_1$ in $G$, that is, $X=\{s_1,s_2,s_1s_2s_1^{-1}\}$. Let $\lambda\in\fie^{\times}$ be a root of $1$, and let 
$V$ be the Yetter--Drinfeld module over $\fie G$ with basis 
$\{v_x:x\in X\}$ and 
\[
xv_y=\lambda v_{xyx^{-1}},\quad x,y\in X.
\]

If $\lambda=-1$, then $\NA(V)$ is finite-dimensional, see Example~\ref{exa:dim12}. 

Note that $X=\supp V$ and that the map
\[ \varphi\colon\Aff(3,-1)\to X,\quad 0\mapsto s_1,\,1\mapsto s_2,\,2\mapsto s_1s_2s_1^{-1}, \] 
is a rack isomorphism. By Lemma~\ref{lem:GXformulas}, $\gamma=\varphi(0)\varphi(1)^{-1}$ generates the derived subgroup of $G$, and has order $3$.

For the following, let $R=\ndZ[\lambda ]$ and let $V_R$ be the $R$-submodule of $V$ generated by $\{v_x:x\in X\}$. Take a maximal ideal $\fm $ of $R$ containing $3$, and define \[ V_{R,\fm}=R/\fm \ot _R V_R. \]

Theorem~\ref{thm:2primes} implies that if the order of $\lambda$ is
divisible by at least two prime factors (in particular, if the order of $\lambda $ is $2\cdot 3^r$ for some $r\ge 1$), then $\dim\NA(V)=\infty$. The proof of Theorem~\ref{thm:2primes} follows the following steps:
\begin{enumerate}
    \item identify a $3$-dimensional Yetter--Drinfeld module $W$ of primitive elements in an associated graded Hopf algebra of $R/\fm \ot_R \NA (V_R)$, spanned by powers of the generators of $V_{R,\fm}$;
    \item use the classification of Nichols algebras of semisimple Yetter--Drinfeld modules to observe that $\NA (V_{R,\fm}\oplus W)$ is infinite-dimensional;
    \item use Lemma~\ref{lem:specialfinite} to conclude that $\NA (V)$ is infinite-dimensional.
\end{enumerate}

Assume now that the order of $\lambda $ in $R$ is different from $2\cdot 3^r$ for all $r\ge 0$, that is, $\lambda 1\ot _R R/\fm \ne -1$.  Then $\NA (V_{R,\fm})$ is infinite-dimensional by Theorem~\ref{thm:char_p}. (For the proof of this theorem one uses the filtration of $V$ induced by the powers of the radical of the derived subgroup of $G$ and the classification of Nichols algebras of diagonal type in prime characteristic.) Consequently, $\NA (V)$ is infinite-dimensional by Lemma~\ref{lem:specialfinite}.

\section{Appendix: On filtrations of braided vector spaces}
\label{ss:appendix}

In this appendix, we determine some consequences of the existence of filtrations
of certain braided vector spaces. We prepare the claims with a lemma.

\begin{lem}
\label{lem:f(i)}
    Let $V$ be a finite-dimensional vector space together with a direct sum decomposition $V=V(1)\oplus V(2)\oplus \cdots \oplus V(l)$ and a flag
    \[
    V=\cF^0V\supseteq \cF^1V \supseteq \cdots \supseteq \cF^mV=\{0\}
    \]
    of subspaces with $l,m\ge 1$.
    Let $n=\dim V$, and for all $1\le i\le n$ let
    \[ f(i)=\max\{ k\ge 0 : i\le \dim \cF^kV \}.
    \]
    Then there exist bases $(x_i)_{1\le i\le n}$ and $(b_i)_{1\le i\le n}$ of $V$, and an upper triangular matrix $S=(s_{ij})\in \fie ^{n\times n}$ with diagonal entries $1$, satisfying all of the following properties:
    \begin{enumerate}
        \item For each $1\le i\le n$ there exists $1\le g(i)\le l$ such that $x_i\in V(g(i))$.
        \item For each $0\le k\le m-1$, the vectors $b_i$ with $1\le i\le \dim \cF^kV$ form a basis of $\cF^kV$.
    \item For all $1\le i\le n$, 
        \[
        b_i=x_i+\sum_{j>i}s_{ij}x_j
        =\displaystyle{x_i+\sum_{\substack{j:f(j)<f(i)\\g(j)\ne g(i)}}s_{ij}x_j}.
        \]
    \end{enumerate}
\end{lem}

\begin{proof}
    The claim and the proof are a variation of the LU decomposition of an invertible square matrix.
    
    Let $\cX=(x_i)_{1\le i\le n}$ and $\cB=(b_i)_{1\le i\le n}$ be 
    bases of $V$ satisfying (1) and (2), respectively.
    Let $S=(s_{ij})\in \fie ^{n\times n}$ be such that $b_i=\sum_{j=1}^ns_{ij}x_j$ for all $1\le i\le n$. Clearly, $S$ is invertible. We are going to modify $\cX$ and $\cB$ step by step such that properties (1) and (2) of the lemma are preserved and $S$ approaches the desired form in (3).

    \textit{Step 1. We may assume that $S$ is an upper triangular matrix with $1$'s on the diagonal, and $s_{ij}\ne 0$ with $i<j$ implies that $f(j)<f(i)$.} Indeed, the first $\dim \cF^{m-1}V$ rows of $S$ are linearly independent. Thus there exist $\dim \cF^{m-1}V$ columns of $S$ such that the corresponding square submatrix has full rank. By permuting the basis vectors of $\cX$, we may assume that these are the first $\dim \cF^{m-1}V$ columns. After applying appropriate row transformations of $S$ (change of vectors $b_i$ in $\cB$ with $f(i)=m-1$) we may assume that $s_{ij}=\delta_{ij}$ (Kronecker's delta) for all $1\le i,j\le \dim \cF^{m-1}V$.

    With the remaining rows of $S$ we proceed by induction. Regarding vectors $b_i\in \cB$ with $f(i)=t$ for some $t$, we first add to them appropriate vectors in $\cF^{t+1}V$ in order to achieve that $s_{ij}=0$ for all $j$ with $f(j)>f(i)$. Then we choose $\dim \cF^tV-\dim \cF^{t+1}V$ vectors from $\cX$ such that the corresponding square submatrix of $S$ (with rows $i$ such that $f(i)=t$) has full rank, and permute them to the columns $j$ with $f(j)=t$. After suitable row transformations in $S$
    (change of vectors $b_i$ in $\cB$ with $f(i)=t$) we may assume that $s_{ij}=\delta_{ij}$ (Kronecker's delta) for all $1\le i,j\le \dim \cF^tV$ with $f(j)\ge f(i)$. We then proceed similarly with the rows $i$ with $f(i)=t-1$.

    \textit{Step 2. We may assume additionally that $s_{ij}=0$ whenever $1\le i<j\le n$, $f(j)<f(i)$, and $g(j)=g(i)$.}
    Indeed, let $i_0$ be the smallest integer such that there exists $j$ with
    $f(j)<f(i_0)$ and $g(j)=g(i_0)$.
    Then replace $x_{i_0}$ by 
    \[
    x_{i_0}+\sum_{\substack{j: f(j)<f(i_0)\\g(j)=g(i_0)}}s_{{i_0}j}x_j.
    \]
    After this transformation, the matrix $S$ does not change in rows $>i_0$, 
    and the basis vectors $b_i$ with $i\le i_0$ will satisfy the required property in (3).
    Thus, by induction on the rows of $S$, we may complete the construction of $\cX$ and $\cB$ with the claimed properties.
\end{proof}

\begin{pro}
\label{pro:nilpotent}
    Let $G$ be a group and $V\in \yd{\fie G}$. Assume that $\dim V<\infty$,
    $V$ is irreducible, and $G$ is generated by the support of $V$. Let $(\cF^iV)_{i\ge 0}$ be a decreasing filtration of the braided vector space $V$ with
    $\cF^1V\ne\{0\}$. Then $h\cF^iV\subseteq \cF^iV$ for all $h\in G$ and $i\ge 0$. Moreover, there exists a 
    normal subgroup $N\ne \{1\}$ of $G$ such that $N\subseteq [G,G]$ and 
    \[ (g-1)\cF^iV\subseteq \cF^{i+1}V
    \]
    for all $g\in N$ and $i\ge 0$. 
\end{pro}

\begin{proof}
  Let $n=\dim V$ and let $\cX=(x_i)_{1\le i\le n}$ and $\cB=(b_i)_{1\le i\le n}$ be bases of $V$ as in Lemma~\ref{lem:f(i)} with respect to the direct sum decomposition of $V$ as a $\fie G$-comodule and the flag of subspaces corresponding to the filtration $(\cF^kV)_{k\ge 0}$. For all $1\le i\le n$ let \[ f(i)=\max\{k\ge 0: i\le \dim \cF^kV\}\]
  as in Lemma~\ref{lem:f(i)}. Then for each $k\ge 0$, the vectors
  \[ (b_i)_{f(i)\ge k} \]
  form a basis of $\cF^kV$.
  For all $1\le i\le n$ let $g_i\in G$ such that $x_i\in V_{g_i}$.
  Let $S=(s_{ij})_{1\le i,j\le n}$ be the upper triangular matrix from Lemma~\ref{lem:f(i)}. We note that
  \begin{enumerate}
      \item[(S1)] $s_{ij}=0$ for all $i,j$ with $i\ne j$ and $f(i)=f(j)$, and
      \item[(S2)] $s_{ij}=0$ for all $i,j$ with $i<j$ and $g(i)=g(j)$.
 \end{enumerate}
  
  Let now $(t_{ij})_{1\le j\le n}$ be the upper triangular matrix with 1's on the diagonal such that
  \[
    x_i=b_i+\sum_{j>i}t_{ij}b_j
  \]
for all $1\le i\le n$.
Then
\begin{align} \label{eq:strel} 
  s_{ij}+t_{ij}+\sum_{i<k<j} s_{ik}t_{kj}=0
\end{align}
  for all $i,j$ with $i<j$.
  Moreover, (S1) implies that
  \begin{enumerate}
      \item[(T1)] $t_{ij}=0$ for all $i,j$ with $i<j$ and $f(i)=f(j)$.
  \end{enumerate}
  
  One obtains for all $1\le i,j\le n$ that 
  \begin{align*}
      c(b_i\ot b_j)&=
      c\Big( \big(x_i+\sum_{k>i}s_{ik}x_k\big)\ot b_j\Big)\\
      &=g_ib_j\ot \Big(b_i+\sum_{k>i}t_{ik}b_k\Big)
      +\sum_{k>i}s_{ik}g_kb_j \ot \Big( b_k+\sum_{l>k}t_{kl}b_l\Big)\\
      &=g_ib_j\ot b_i+\sum_{k>i}t_{ik}g_i b_j\ot b_k
      +\sum_{k>i}s_{ik}g_kb_j \ot b_k
      +\sum_{k>l>i}s_{il}t_{lk} g_lb_j\ot b_k.
\end{align*}
We can slightly reformulate the last term by inserting and subtracting additional terms. We obtain that
\begin{align*}
      c(b_i\ot b_j)=
      g_ib_j\ot b_i&+\sum_{k>i}(t_{ik}+s_{ik})g_i b_j\ot b_k
      +\sum_{k>i}s_{ik}(g_kb_j-g_ib_j)\ot b_k\\
      &+\sum_{k>l>i}s_{il}t_{lk} (g_lb_j-g_ib_j)\ot b_k
      +\sum_{k>l>i}s_{il}t_{lk} g_ib_j\ot b_k.
\end{align*}
Then Equation~\eqref{eq:strel} gives the formula
\begin{equation}
\label{eq:cbb}
\begin{aligned}
 c(b_i\ot b_j) =
      g_ib_j\ot b_i
      &+\sum_{k>i}s_{ik}(g_kb_j-g_ib_j)\ot b_k\\
      &+\sum_{k>l>i}s_{il}t_{lk} (g_lb_j-g_ib_j)\ot b_k.
\end{aligned}
\end{equation}

Recall that for all $m\ge 0$, the vectors $(b_\alpha)_{f(\alpha )\ge m}$ form a basis of $\cF^mV$. Hence for all $m\ge 0$, the vectors $b_\alpha \ot b_\beta $ with $f(\alpha )+f(\beta)\ge m$
form a basis of the vector space $\sum_{k+l\ge m}\cF^kV\ot \cF^lV$.

Now take $1\le i,j\le n$. Then $b_i\ot b_j\in \cF^{f(i)}V\ot \cF^{f(j)}V$. Let $m=f(i)+f(j)$. Since $(\cF^kV)_{k\ge 0}$ is a decreasing filtration of the braided vector space $V$, we conclude that
$c(b_i\ot b_j)$ is a linear combination of the tensors
$b_\alpha \ot b_\beta $ with $f(\alpha)+f(\beta)\ge m$.
Then Equation~\eqref{eq:cbb} and (S1) imply that $g_ib_j\in \cF^{f(j)}V$, and since $G$ is generated by the elements $g_k$ with $1\le k\le n$, we conclude that
\begin{enumerate}
    \item[(C1)] for all $k\ge 0$, $\cF^kV$ is a $\fie G$-submodule of $V$.
 \end{enumerate}
 
It also follows from the previous paragraph that
\begin{enumerate}
    \item[(S3)] there exist $i<j$ with $s_{ij}\ne 0$.
\end{enumerate}
Ideed, if $s_{ij}=0$ for all $i<j$, then $b_i\in V_{g_i}$ for all $1\le i\le n$. In particular, $\cF^iV$ is a $\fie G$-subcomodule for all $i$, and hence a Yetter--Drinfeld submodule because of (C1). However, this contradicts the irreducibility of $V$ and 
the fact that $\cF^1V\ne\{0\}$.

Now let $d=\min \{ f(i)-f(j): i<j,s_{ij}\ne 0\}$. Then, by (T1) and Equation~\eqref{eq:strel}, it follows that
\begin{enumerate}
    \item[(T2)] $t_{ij}=0$ for all $i<j$ with $f(i)-f(j)<d$.
\end{enumerate}
Now the assumptions on $c$ and Equation~\eqref{eq:cbb} imply that
\[ s_{ik}(g_kb_j-g_ib_j)\in \cF^{f(j)+d} \quad
\text{for all $i<k$ with $f(i)-f(k)=d$.} 
\]
Using (S2) and (S3), it follows that
    \begin{enumerate}
    \item[(C2)] there exist $i<j$ such that $s_{ij}\ne 0$, $g_i\ne g_j$, and $g_i^{-1}g_jv-v\in \cF^{l+1}V$ for all $l\ge 0$ and $v\in \cF^lV$.
\end{enumerate}
Since $g_i^{-1}g_j\in [G,G]$ for all $1\le i<j\le n$, (C2) implies that the normal subgroup $N$ of $G$ generated by $g_i^{-1}g_j$ in (C2) satisfies the properties required in the proposition.
This completes the proof.
\end{proof}

\begin{cor}
\label{cor:simples}
    Let $G$ be a finite group generated by a conjugacy class $X$. Assume that
    the derived subgroup $[G,G]$ is simple non-abelian. Let $V\in\yd{\fie G}$ be 
    irreducible with support $X$. Then $V$ does not admit a non-trivial 
    decreasing filtration.  
\end{cor}

\begin{proof}
    Assume that $V$ admits a non-trivial decreasing filtration. 
    By Proposition~\ref{pro:nilpotent}, there exists a non-trivial 
    normal subgroup $N$ of $G$ such that $N\subseteq [G,G]$ and 
    \begin{align} \label{eq:g-1F}  
(g-1)\cF^iV\subseteq \cF^{i+1}V
    \end{align}
    for all $g\in N$ and $i\ge 0$. 
    Then $N=[G,G]$ since $[G,G]$ is simple. Since $[G,G]$ is non-abelian, 
    there exists $1\ne g\in [G,G]$
    of order $m$ which is coprime to $\ch(\fie)$. In particular, 
    $\ch(\fie)=p>0$. Then $(g-1)^{p^k}=g^{p^k}-1$ for all $k\ge 1$, and $g^{p^k}=g$ for $k$ the order of $p$ in $U(\ndZ/m\ndZ)$. Thus Equation~\eqref{eq:g-1F} implies that $(g-1)v=0$ for all $v\in V$. The set of all $h\in [G,G]$ acting trivially on $V$ is a normal subgroup of $G$ and has then to coincide with $[G,G]$. Thus the action of $G$ on $V$ factors through the abelian group $G/[G,G]$, and hence $|X|=1$, a contradiction to the assumptions on $G$.
\end{proof}

We conclude the paper with some questions.

Following a question of Andruskiewitsch, we remark that with our techniques
the Gelfand--Kirillov dimension of the Nichols algebras in Theorem~\ref{thm:main}
can not yet be determined.
The main reason for this is that, at the moment, no sufficiently strong results are available on the Gelfand--Kirillov dimension
of Nichols algebras of diagonal type in positive characteristic.

\begin{question}
Is it possible to determine the precise Hilbert series and the Gelfand--Kirillov dimension of the Nichols algebra of a Yetter--Drinfeld module as in Corollary~\ref{cor:Vp}?
\end{question}

A basis for the Nichols algebra of Example~\ref{exa:dim12}
can be obtained by a straightforward calculation using the Diamond
Lemma. However, computer calculations are needed to 
obtain bases for the algebras of the other examples mentioned in the introduction. 
\begin{question}
    Is it possible to construct without computer calculations a basis for each of the Nichols algebras in Examples~\ref{exa:dim1280a}--\ref{exa:dim326592b}?
\end{question}

\subsection*{Acknowledgements}

This work was partially supported by
the project OZR3762 of Vrije Universiteit Brussel.
EM would like to thank Ben Martin for fruitful discussions about geometric invariant theory in positive characteristic. 
\bibliographystyle{abbrv} 
\bibliography{refs.bib}

\begin{thebibliography}{10}

\bibitem{MR3728608}
N.~Andruskiewitsch.
\newblock On finite-dimensional {H}opf algebras.
\newblock In {\em Proceedings of the {I}nternational {C}ongress of
  {M}athematicians---{S}eoul 2014. {V}ol. {II}}, pages 117--141. Kyung Moon Sa,
  Seoul, 2014.

\bibitem{MR3751453}
N.~Andruskiewitsch.
\newblock An introduction to {N}ichols algebras.
\newblock In {\em Quantization, geometry and noncommutative structures in
  mathematics and physics}, Math. Phys. Stud., pages 135--195. Springer, Cham,
  2017.

\bibitem{MR4298502}
N.~Andruskiewitsch, I.~Angiono, and I.~Heckenberger.
\newblock On finite {GK}-dimensional {N}ichols algebras over abelian groups.
\newblock {\em Mem. Amer. Math. Soc.}, 271(1329):ix+125, 2021.

\bibitem{poisson}
N.~Andruskiewitsch, I.~Angiono, and M.~Yakimov.
\newblock Poisson orders on large quantum groups.
\newblock {\em \verb+arXiv:2008.11025+}, 2023.

\bibitem{MR3395052}
N.~Andruskiewitsch, G.~Carnovale, and G.~A. Garc\'{\i}a.
\newblock Finite-dimensional pointed {H}opf algebras over finite simple groups
  of {L}ie type {I}. {N}on-semisimple classes in {${\bf PSL}_n(q)$}.
\newblock {\em J. Algebra}, 442:36--65, 2015.

\bibitem{MR3493214}
N.~Andruskiewitsch, G.~Carnovale, and G.~A. Garc\'{\i}a.
\newblock Finite-dimensional pointed {H}opf algebras over finite simple groups
  of {L}ie type {II}: unipotent classes in symplectic groups.
\newblock {\em Commun. Contemp. Math.}, 18(4):1550053, 35, 2016.

\bibitem{MR3713037}
N.~Andruskiewitsch, G.~Carnovale, and G.~A. Garc\'{\i}a.
\newblock Finite-dimensional pointed {H}opf algebras over finite simple groups
  of {L}ie type {III}. {S}emisimple classes in {${\bf PSL}_n(q)$}.
\newblock {\em Rev. Mat. Iberoam.}, 33(3):995--1024, 2017.

\bibitem{MR4109132}
N.~Andruskiewitsch, G.~Carnovale, and G.~A. Garc\'{\i}a.
\newblock Finite-dimensional pointed {H}opf algebras over finite simple groups
  of {L}ie type {IV}. {U}nipotent classes in {C}hevalley and {S}teinberg
  groups.
\newblock {\em Algebr. Represent. Theory}, 23(3):621--655, 2020.

\bibitem{MR4325965}
N.~Andruskiewitsch, G.~Carnovale, and G.~A. Garc\'{\i}a.
\newblock Finite-dimensional pointed {H}opf algebras over finite simple groups
  of {L}ie type {V}. {M}ixed classes in {C}hevalley and {S}teinberg groups.
\newblock {\em Manuscripta Math.}, 166(3-4):605--647, 2021.

\bibitem{MR2799090}
N.~Andruskiewitsch, F.~Fantino, G.~A. Garc\'{\i}a, and L.~Vendramin.
\newblock On {N}ichols algebras associated to simple racks.
\newblock In {\em Groups, algebras and applications}, volume 537 of {\em
  Contemp. Math.}, pages 31--56. Amer. Math. Soc., Providence, RI, 2011.

\bibitem{MR2786171}
N.~Andruskiewitsch, F.~Fantino, M.~Gra\~{n}a, and L.~Vendramin.
\newblock Finite-dimensional pointed {H}opf algebras with alternating groups
  are trivial.
\newblock {\em Ann. Mat. Pura Appl. (4)}, 190(2):225--245, 2011.

\bibitem{MR2745542}
N.~Andruskiewitsch, F.~Fantino, M.~Gra\~{n}a, and L.~Vendramin.
\newblock Pointed {H}opf algebras over the sporadic simple groups.
\newblock {\em J. Algebra}, 325:305--320, 2011.

\bibitem{MR1714540}
N.~Andruskiewitsch and M.~Gra\~{n}a.
\newblock Braided {H}opf algebras over non-abelian finite groups.
\newblock volume~63, pages 45--78. 1999.
\newblock Colloquium on Operator Algebras and Quantum Groups (Spanish)
  (Vaquer\'{\i}as, 1997).

\bibitem{MR1994219}
N.~Andruskiewitsch and M.~Gra\~{n}a.
\newblock From racks to pointed {H}opf algebras.
\newblock {\em Adv. Math.}, 178(2):177--243, 2003.

\bibitem{MR2766176}
N.~Andruskiewitsch, I.~Heckenberger, and H.-J. Schneider.
\newblock The {N}ichols algebra of a semisimple {Y}etter-{D}rinfeld module.
\newblock {\em Amer. J. Math.}, 132(6):1493--1547, 2010.

\bibitem{MR1659895}
N.~Andruskiewitsch and H.-J. Schneider.
\newblock Lifting of quantum linear spaces and pointed {H}opf algebras of order
  {$p^3$}.
\newblock {\em J. Algebra}, 209(2):658--691, 1998.

\bibitem{MR1913436}
N.~Andruskiewitsch and H.-J. Schneider.
\newblock Pointed {H}opf algebras.
\newblock In {\em New directions in {H}opf algebras}, volume~43 of {\em Math.
  Sci. Res. Inst. Publ.}, pages 1--68. Cambridge Univ. Press, Cambridge, 2002.

\bibitem{MR2209265}
Y.~Bazlov.
\newblock Nichols-{W}oronowicz algebra model for {S}chubert calculus on
  {C}oxeter groups.
\newblock {\em J. Algebra}, 297(2):372--399, 2006.

\bibitem{MR3981991}
A.~Berenstein and D.~Kazhdan.
\newblock Hecke-{H}opf algebras.
\newblock {\em Adv. Math.}, 353:312--395, 2019.

\bibitem{MR3552907}
J.~Blasiak, R.~I. Liu, and K.~M\'{e}sz\'{a}ros.
\newblock Subalgebras of the {F}omin-{K}irillov algebra.
\newblock {\em J. Algebraic Combin.}, 44(3):785--829, 2016.

\bibitem{MR1727221}
N.~Bourbaki.
\newblock {\em Commutative algebra. {C}hapters 1--7}.
\newblock Elements of Mathematics (Berlin). Springer-Verlag, Berlin, 1998.
\newblock Translated from the French, Reprint of the 1989 English translation.

\bibitem{MR4151588}
G.~Carnovale and M.~Costantini.
\newblock Finite-dimensional pointed {H}opf algebras over finite simple groups
  of {L}ie type {VI}. {S}uzuki and {R}ee groups.
\newblock {\em J. Pure Appl. Algebra}, 225(4):Paper No. 106568, 19, 2021.

\bibitem{CM1}
J.~Cuadra and E.~Meir.
\newblock On the existence of orders in semisimple {H}opf algebras.
\newblock {\em Trans. Amer. Math. Soc.}, 368(4):2547--2562, 2016.

\bibitem{ETW}
J.~S. Ellenberg, T.~Tran, and C.~Westerland.
\newblock Fox-neuwirth-fuks cells, quantum shuffle algebras, and malle's
  conjecture for function fields.
\newblock {\em \verb+arXiv:1701.0454+}, 2023.

\bibitem{MR1848966}
P.~Etingof, A.~Soloviev, and R.~Guralnick.
\newblock Indecomposable set-theoretical solutions to the quantum
  {Y}ang-{B}axter equation on a set with a prime number of elements.
\newblock {\em J. Algebra}, 242(2):709--719, 2001.

\bibitem{MR3077241}
F.~Fantino and L.~Vendramin.
\newblock On twisted conjugacy classes of type {D} in sporadic simple groups.
\newblock In {\em Hopf algebras and tensor categories}, volume 585 of {\em
  Contemp. Math.}, pages 247--259. Amer. Math. Soc., Providence, RI, 2013.

\bibitem{MR1667680}
S.~Fomin and A.~N. Kirillov.
\newblock Quadratic algebras, {D}unkl elements, and {S}chubert calculus.
\newblock In {\em Advances in geometry}, volume 172 of {\em Progr. Math.},
  pages 147--182. Birkh\"{a}user Boston, Boston, MA, 1999.

\bibitem{MR2803792}
M.~Gra\~{n}a, I.~Heckenberger, and L.~Vendramin.
\newblock Nichols algebras of group type with many quadratic relations.
\newblock {\em Adv. Math.}, 227(5):1956--1989, 2011.

\bibitem{MR2207786}
I.~Heckenberger.
\newblock The {W}eyl groupoid of a {N}ichols algebra of diagonal type.
\newblock {\em Invent. Math.}, 164(1):175--188, 2006.

\bibitem{MR2462836}
I.~Heckenberger.
\newblock Classification of arithmetic root systems.
\newblock {\em Adv. Math.}, 220(1):59--124, 2009.

\bibitem{MR2891215}
I.~Heckenberger, A.~Lochmann, and L.~Vendramin.
\newblock Braided racks, {H}urwitz actions and {N}ichols algebras with many
  cubic relations.
\newblock {\em Transform. Groups}, 17(1):157--194, 2012.

\bibitem{MR3356939}
I.~Heckenberger, A.~Lochmann, and L.~Vendramin.
\newblock Nichols algebras with many cubic relations.
\newblock {\em Trans. Amer. Math. Soc.}, 367(9):6315--6356, 2015.

\bibitem{MR4164719}
I.~Heckenberger and H.-J. Schneider.
\newblock {\em Hopf algebras and root systems}, volume 247 of {\em Mathematical
  Surveys and Monographs}.
\newblock American Mathematical Society, Providence, RI, [2020] \copyright
  2020.

\bibitem{MR3605018}
I.~Heckenberger and L.~Vendramin.
\newblock A classification of {N}ichols algebras of semisimple
  {Y}etter-{D}rinfeld modules over non-abelian groups.
\newblock {\em J. Eur. Math. Soc. (JEMS)}, 19(2):299--356, 2017.

\bibitem{MR3656477}
I.~Heckenberger and L.~Vendramin.
\newblock The classification of {N}ichols algebras over groups with finite root
  system of rank two.
\newblock {\em J. Eur. Math. Soc. (JEMS)}, 19(7):1977--2017, 2017.

\bibitem{MR4034793}
I.~Heckenberger and L.~Vendramin.
\newblock P{BW} deformations of a {F}omin-{K}irillov algebra and other
  examples.
\newblock {\em Algebr. Represent. Theory}, 22(6):1513--1532, 2019.

\bibitem{MR2426855}
I.~M. Isaacs.
\newblock {\em Finite group theory}, volume~92 of {\em Graduate Studies in
  Mathematics}.
\newblock American Mathematical Society, Providence, RI, 2008.

\bibitem{MR4176533}
M.~Kapranov and V.~Schechtman.
\newblock Shuffle algebras and perverse sheaves.
\newblock {\em Pure Appl. Math. Q.}, 16(3):573--657, 2020.

\bibitem{MR1321145}
C.~Kassel.
\newblock {\em Quantum groups}, volume 155 of {\em Graduate Texts in
  Mathematics}.
\newblock Springer-Verlag, New York, 1995.

\bibitem{MR1763385}
V.~K. Kharchenko.
\newblock A quantum analogue of the {P}oincar\'{e}-{B}irkhoff-{W}itt theorem.
\newblock {\em Algebra Log.}, 38(4):476--507, 509, 1999.

\bibitem{MR3253277}
S.~Lentner.
\newblock New large-rank {N}ichols algebras over nonabelian groups with
  commutator subgroup {$\mathbb{Z}_2$}.
\newblock {\em J. Algebra}, 419:1--33, 2014.

\bibitem{MR4184294}
S.~D. Lentner.
\newblock Quantum groups and {N}ichols algebras acting on conformal field
  theories.
\newblock {\em Adv. Math.}, 378:Paper No. 107517, 71, 2021.

\bibitem{MR1969778}
S.~Majid and E.~Raineri.
\newblock Electromagnetism and gauge theory on the permutation group {$S_3$}.
\newblock {\em J. Geom. Phys.}, 44(2-3):129--155, 2002.

\bibitem{MR4401961}
E.~Meir.
\newblock Geometric perspective on {N}ichols algebras.
\newblock {\em J. Algebra}, 601:390--422, 2022.

\bibitem{MR1800714}
A.~Milinski and H.-J. Schneider.
\newblock Pointed indecomposable {H}opf algebras over {C}oxeter groups.
\newblock In {\em New trends in {H}opf algebra theory ({L}a {F}alda, 1999)},
  volume 267 of {\em Contemp. Math.}, pages 215--236. Amer. Math. Soc.,
  Providence, RI, 2000.

\bibitem{MR506406}
W.~D. Nichols.
\newblock Bialgebras of type one.
\newblock {\em Comm. Algebra}, 6(15):1521--1552, 1978.

\bibitem{MR1632802}
M.~Rosso.
\newblock Quantum groups and quantum shuffles.
\newblock {\em Invent. Math.}, 133(2):399--416, 1998.

\bibitem{MR1396857}
P.~Schauenburg.
\newblock A characterization of the {B}orel-like subalgebras of quantum
  enveloping algebras.
\newblock {\em Comm. Algebra}, 24(9):2811--2823, 1996.

\bibitem{MR0252485}
M.~E. Sweedler.
\newblock {\em Hopf algebras}.
\newblock Mathematics Lecture Note Series. W. A. Benjamin, Inc., New York,
  1969.

\bibitem{MR3313687}
J.~Wang and I.~Heckenberger.
\newblock Rank 2 {N}ichols algebras of diagonal type over fields of positive
  characteristic.
\newblock {\em SIGMA Symmetry Integrability Geom. Methods Appl.}, 11:Paper 011,
  24, 2015.

\bibitem{MR901157}
S.~L. Woronowicz.
\newblock Compact matrix pseudogroups.
\newblock {\em Comm. Math. Phys.}, 111(4):613--665, 1987.

\bibitem{MR994499}
S.~L. Woronowicz.
\newblock Differential calculus on compact matrix pseudogroups (quantum
  groups).
\newblock {\em Comm. Math. Phys.}, 122(1):125--170, 1989.

\end{thebibliography}

\end{document}